\documentclass[11pt]{amsart}
\usepackage{amsmath}%
\usepackage{amsfonts}%
\usepackage{amssymb}%
\usepackage{graphicx}%
\usepackage{mathrsfs}
\usepackage{mathtools}%
\usepackage{xcolor}%
\usepackage{nameref,hyperref,url}
\usepackage[top=1in, bottom=1.25in, left=1in, right=1in]{geometry}
\usepackage[foot]{amsaddr}
\newtheorem{theorem}{Theorem}
\theoremstyle{plain}

\newtheorem{definition}{Definition}

\newtheorem{lemma}{Lemma}

\newtheorem{proposition}{Proposition}
\newtheorem{remark}{Remark}

\numberwithin{equation}{section}

\begin{document}

\title[Fractional Schr\"odinger systems coupled by Hardy-Sobolev critical terms]
{Fractional Schr\"odinger systems coupled by Hardy-Sobolev critical terms}

\author{Alejandro Ortega}
\email[A. Ortega ]{alortega@math.uc3m.es}
\address[A. Ortega]{Departamento de Matem\'aticas,  
Universidad Carlos III de Madrid, Av. Universidad 30, 28911 Legan\'es (Madrid), Spain}

\thanks{The author is partially supported by the State Research Agency of Spain, under research project PID2019-106122GB-I00.}

\begin{abstract}
In this work we analyze a class of nonlinear fractional elliptic systems involving Hardy--type potentials and coupled by critical Hardy-Sobolev--type nonlinearities in $\mathbb{R}^N$. Due to the lack of compactness at the critical exponent the variational approach requires a careful analysis of the Palais-Smale sequences. In order to overcome this loss of compactness, by means of a concentration--compactness argument the compactness of PS sequences is derived. This, combined with a energy characterization of the semi-trivial solutions, allow us to conclude the existence of positive ground and bound state solutions in terms of coupling parameter $\nu>0$ and the involved exponents $\alpha,\beta$.
\end{abstract}
\maketitle


\section{Introduction}
In this work we analyze the existence of positive bound and ground state solutions for a class of nonlinear fractional elliptic systems involving Hardy--type potentials and coupled by critical Hardy-Sobolev--type nonlinearities, namely
\begin{equation}\label{system}\tag{$S_{\alpha,\beta}$}
\left\{\begin{array}{ll}
\displaystyle (-\Delta)^{s} u - \lambda_1 \frac{u}{|x|^{2s}}-\frac{u^{2^*_{s}(t)-1}}{|x|^{t}}= \nu \alpha h(x) \frac{u^{\alpha-1}v^\beta}{|x|^{t}}   &\text{in }\mathbb{R}^N,\vspace{.3cm}\\
\displaystyle (-\Delta)^{s} v - \lambda_2 \frac{v}{|x|^{2s}}-\frac{v^{2^*_{s}(t)-1}}{|x|^{t}}= \nu \beta h(x) \frac{u^\alpha v^{\beta-1}}{|x|^{t}} &\text{in }\mathbb{R}^N,\vspace{.3cm}\\
u,v> 0 & \text{in }\mathbb{R}^N\setminus\{0\},
\end{array}\right.
\end{equation}
where $0<s<1$, $0< t<2s$ and $2_s^*(t)=\frac{2(N-t)}{N-2s}$ the fractional Hardy-Sobolev critical exponent, $\lambda_1,\lambda_2\in(0,\Lambda_{N,s})$ with $\Lambda_{N,s}=2\pi^{\frac{N}{2}}\frac{\Gamma^2\left(\frac{N+2s}{4}\right)}{\Gamma^2\left(\frac{N-2s}{4}\right)}\frac{\Gamma\left(\frac{N+2s}{2}\right)}{|\Gamma(-s)|}$ the best Hardy constant in the fractional Hardy inequality, the coupling parameter $\nu>0$ and the real parameters $\alpha,\beta$ such that
\begin{equation}\label{alphabeta}
\alpha, \beta> 1 \qquad \mbox{ and } \qquad  \alpha+\beta\le 2_{s}^*(t),
\end{equation}
whereas the function $h\gneq0$ is such that  {
\begin{equation}\label{H}\tag{$H$}
h(x)\in L^{\frac{2^*_s(t)}{2^*_s(t)-\alpha-\beta},t}(\mathbb{R}^N),
\end{equation}
where $L^{p,t}(\mathbb{R}^N)$ denotes the weighted $L^p$-based space of measurable functions $u$ such that
\[\| u\|_{L^{p,t}(\mathbb{R}^N)}^p:= \int_{\mathbb{R}^N}\frac{|u|^p}{|x|^t}dx<\infty.\]
In the case of having $\alpha+\beta=2^*_s(t)$, then \eqref{H} simply establishes $h\in L^{\infty}(\mathbb{R}^N)$.
}
Here $(-\Delta)^s$ denotes the fractional Laplace operator,
\begin{equation*}
(-\Delta)^su(x)=p.v.\int_{\mathbb{R}^N}\frac{u(x)-u(y)}{|x-y|^{N+2s}}dy.
\end{equation*}
Systems like \eqref{system} has attracted great attention in the last years. For $s=1$ and $t=0$, we get the system 
\begin{equation}\label{local}
\left\{\begin{array}{ll}
\displaystyle -\Delta u - \lambda_1 \frac{u}{|x|^{2}}-u^{2^*-1}= \nu \alpha h(x) u^{\alpha-1}v^\beta   &\text{in }\mathbb{R}^N,\vspace{.3cm}\\
\displaystyle -\Delta v - \lambda_2 \frac{v}{|x|^{2}}-v^{2^*-1}= \nu \beta h(x) u^\alpha v^{\beta-1} &\text{in }\mathbb{R}^N,\vspace{.3cm}\\
u,v> 0 & \text{in }\mathbb{R}^N\setminus\{0\},
\end{array}\right.
\end{equation}
which arises in connection with solitary or standing wave solutions to Gross-Pitaevskii--type systems. This problem occurs in various physical phenomena, like the Hartree-Fock theory for double condensates, binary mixture Bose-Einstein condensates, and nonlinear optics (cf. \cite{Akhmediev1999, Esry1997, Frantzeskakis2010,Kivshar1998}).  The existence of positive ground state and bound state solutions to \eqref{local} was addressed in \cite{Abdellaoui2009} where the authors proved the existence of positive ground state solutions depending on the parameter $\nu>0$ under the assumption $0\lneq h\in L^\infty(\mathbb{R}^N)\cap L^1(\mathbb{R}^N)$ if $2<\alpha+\beta<2_1^*(0)$ and $h\in L^\infty(\mathbb{R}^N)$ if $2<\alpha+\beta=2_1^*(0)$. This was then extended by relaxing the hypotheses on $h$ (cf. \cite{Kang2014, Zhong2015}), by considering a wider range on the exponents (cf. \cite{Colorado2022}), by considering extremal values of the exponents, namely $\alpha=2,\beta=1$ (cf. \cite{Colorado2023}), or by introducing Hardy-Sobolev potentials in the coupling term (cf. \cite{LopezSoriano2023}). Precisely, the main aim of his work is to extend to the fractional setting the results of \cite{LopezSoriano2023}. { Moreover, regarding the condition \eqref{H}, let us stress the following. As it is noted in \cite{LopezSoriano2023}, for the case $s=1$ and $t=0$, studied in \cite{Abdellaoui2009,Colorado2022,Colorado2023}, it is assumed that $h\in L^1(\mathbb{R}^N)\cap L^\infty(\mathbb{R}^N)$ if $\alpha+\beta<2^*$, which  in turn also implies that $h\in L^q(\mathbb{R}^N)$ for every $q>1$. In contrast, the hypotheses \eqref{H} relaxes such condition by choosing the appropriate intermediate functional space that still allows $h$ to control de concentration phenomena at $0$ and $\infty$, namely, $h\in L^{\frac{2^*}{2^*-\alpha-\beta}}(\mathbb{R}^N)$ if $s=1$ and $t=0$ or, in general, $h\in L^{\frac{2_s^*(t)}{2_s^*(t)-\alpha-\beta},t}(\mathbb{R}^N)$ if $0<s<1$ and $t>0$.}

The study of systems like \eqref{system} continues an extensive line of research on elliptic systems. For instance, the above system \eqref{local} is closely related to the Schr\"{o}dinger--type system
\begin{equation}\label{Sch}
\left\{\begin{array}{ll}
\displaystyle -\Delta u - V_1(x) u=\mu_1u^{2p-1} +\nu  u^{p-1}v^p   &\text{in }\mathbb{R}^N,\vspace{.3cm}\\
\displaystyle -\Delta v - V_2(x)v=\mu_2 v^{2^*-1}+ \nu  u^p v^{p-1} &\text{in }\mathbb{R}^N,\vspace{.3cm}\\
u,v> 0 & \text{in }\mathbb{R}^N\setminus\{0\},
\end{array}\right.
\end{equation}
where $2\leq 2p\leq 2_1^*(0)$. Existence and multiplicity of solutions to \eqref{Sch} has been extensively addressed in the literature (cf. \cite{Ambrosetti2007, Ambrosetti2007a, Bartsch2006, Lin2005,Liu2010,Maia2006,Pomponio2006} and references therein). Consequently, for general $0<s<1$ and $t=0$ the system \ref{system} is closely related to the fractional Schr\"{o}dinger system
\begin{equation}\label{FSch}
\left\{\begin{array}{ll}
\displaystyle (-\Delta)^s u - V_1(x) u=\mu_1u^{2p-1} +\nu  u^{p-1}v^p   &\text{in }\mathbb{R}^N,\vspace{.3cm}\\
\displaystyle (-\Delta)^s v - V_2(x)v=\mu_2 v^{2p-1}+ \nu  u^p v^{p-1} &\text{in }\mathbb{R}^N,\vspace{.3cm}\\
u,v> 0 & \text{in }\mathbb{R}^N\setminus\{0\},
\end{array}\right.
\end{equation}
where $2\leq 2p\leq 2_s^*(0)$. Existence and multiplicity of solutions to the fractional elliptic system \eqref{FSch} or related ones can be found in \cite{Colorado2023a,Guo2016,He2017,He2020,Shen2022,Xu2022} and references therein. 

An important feature of system \eqref{system} is that, either we consider the uncoupled case $\nu=0$ or one component vanish, we are left to deal with the solutions $z^{\lambda,t,s}$ to the equation, 
\begin{equation}\label{e1}
 (-\Delta)^{s} z - \lambda \frac{z}{|x|^{2s}}=\frac{z^{2^*_{s}(t)-1}}{|x|^{t}},\qquad \text{and}\ z>0\ \text{in}\ \mathbb{R}^N\backslash\{0\}.
\end{equation}
In the case $s=1$, solutions $z^{\lambda,t,1}$ of the above equation were completely characterized by Kang \cite{Kang2004}. In particular, the family of the rescaled functions associated to $z^{\lambda,t,1}$ turns out to be the extremal functions for the corresponding Hardy-Sobolev constant (see \eqref{S1} with $s=1$) or, equivalenty, the extremal functions for the Hardy-Sobolev inequality (see \ref{HS_ineq_l} with $s=1$). In addition, explicit expressions for that solutions (see \eqref{Keq}) are also available. On the contrary, no such explicit expression is known for solutions $z^{\lambda,t,s}$ to \eqref{e1} and general $0<s<1$. Nevertheless, some important properties has been recently proved for solutions to \eqref{e1} for $0<s<1$ and $0\leq t<2s$ (cf. \cite{Dipierro2016,Ghoussoub2018,Ghoussoub2015}. These results, together with the concentration properties and concentration profile characterization of $z^{\lambda,t,s}$ provided in \cite{Bhakta2021} and \cite{Bonder2018, Chen2018}, play a key role in the variational argument used in this work as they represent the fractional counterpart to the the classical results on concentration profile characterization \cite{Li2011,Smets2005,Struwe2008} and the well-known Concentration-Compactness Principle by P.-L. Lions \cite{Lions1,Lions2} respectively. Recently, in \cite{Kumar2022} the authors addressed system \eqref{system} with $t=0$. It is worth to mention that, in case of having $t=0$, the concentration phenomena can take place at different points $x_j\in\mathbb{R}^N$ which requires additional hypotheses, either on the size of the coupling parameter or by assuming the function $h$ to be a radial function, in order to avoid such phenomena.\newline

\vspace{0.3cm}

{\bf Organization of the paper:} In Section \ref{sect:Functional}, we outline the functional settings and we present some preliminary results concerning our variational approach based on the Nehari manifold. Additionally, we explore the local behavior of the semi-trivial solutions in terms of the coupling parameter $\nu>0$ and the exponents $\alpha,\beta$. In Section \ref{sect:Palais-Smale}, we examine the Palais-Smale condition for the energy functional associated to \ref{system}, which is proved to be satisfied in both subcritical and critical regimes. In this step, we make a fundamental use of the concentration results

 Finally, in Section \ref{sect:Main-Results}, we provide the proof for the main results about existence of positive bound and ground state solutions to system \ref{system}.

\section{Functional Settings and Preliminaries}\label{sect:Functional}
The system \eqref{system} is the Euler-Lagrange system corresponding to energy functional 
\begin{equation*}
\begin{split}
J_\nu(u,v)=&\frac{1}{2}\iint_{\mathbb{R}^{2N}}\frac{|u(x)-u(y)|^2}{|x-y|^{N+2s}}dydx+\frac{1}{2}\iint_{\mathbb{R}^{2N}}\frac{|v(x)-v(y)|^2}{|x-y|^{N+2s}}dydx\\
&-\frac{\lambda_1}{2}\int_{\mathbb{R}^N}\frac{u^2}{|x|^{2s}}dx-\frac{\lambda_2}{2}\int_{\mathbb{R}^N}\frac{v^2}{|x|^{2s}}dx-\frac{1}{2_s^*(t)}\int_{\mathbb{R}^N}\frac{|u|^{2_s^*(t)}}{|x|^{t}}dx-\frac{1}{2_s^*(t)}\int_{\mathbb{R}^N}\frac{|v|^{2_s^*(t)}}{|x|^{t}}dx\\
&-\nu\int_{\mathbb{R}}h(x)\frac{|u|^\alpha|v|^\beta}{|x|^{t}}dx.
\end{split}
\end{equation*}
The functional $J_\nu$ is well defined on the space $\mathbb{D}\vcentcolon=\mathcal{D}^{s,2}(\mathbb{R}^{N})\times \mathcal{D}^{s,2}(\mathbb{R}^{N})$ with $\mathcal{D}^{s,2}(\mathbb{R}^{N})$ being the closure of $C_0^{\infty}(\mathbb{R}^N)$ under the Gagliardo seminorm
\begin{equation*}
\| u\|_{\mathcal{D}^{s,2}(\mathbb{R}^{N})}^2\vcentcolon= \iint_{\mathbb{R}^{2N}}\frac{|u(x)-u(y)|^2}{|x-y|^{N+2s}}dydx.
\end{equation*}
We refer to \cite{Brasco2021,Brasco2019} for a detailed analysis of the spaces $\mathcal{D}^{s,2}(\mathbb{R}^{N})$. Moreover, it is easy to see that the functional $\mathcal{J}_\nu$ is of class $C^1(\mathbb{D},\mathbb{R})$. The space $\mathbb{D}$ is endowed with the norm
\begin{equation*}
\|(u,v)\|_{\mathbb{D}}^2=\|u\|_{\lambda_1,s}^2+ \|v\|_{\lambda_2,s}^2
\end{equation*}
where
\begin{equation*}
\|u\|_{\lambda,s}^2=\iint_{\mathbb{R}^{2N}}\frac{|u(x)-u(y)|^2}{|x-y|^{N+2s}}dydx-\lambda\int_{\mathbb{R}^N}\frac{u^2}{|x|^{2s}}dx.
\end{equation*}
Note that, because of the fractional Hardy inequality (cf. \cite{Frank2008}),
\begin{equation}\label{hardy_inequality}
\Lambda_{N,s}\int_{\mathbb{R}^N}\frac{u^2}{|x|^{2s}}dx\leq \iint_{\mathbb{R}^{2N}}\frac{|u(x)-u(y)|^2}{|x-y|^{N+2s}}dydx,
\end{equation}
the norm $\|\cdot\|_{\lambda,s}$ is well defined. Moreover, due to \eqref{hardy_inequality} both norms, $\| \cdot \|_{\mathcal{D}^{s,2}(\mathbb{R}^{N})}$ and $\|\cdot\|_{\lambda,s}$ are equivalent norms provided $\lambda\in(0,\Lambda_{N,s})$.

On the other hand, due to the homogeneities of the terms composing the functional $J_\nu$, it is easy to see that $\mathcal{J}_\nu(t \tilde u,t \tilde v) \to -\infty$ as $t\to+\infty$ so that the functional $J_\nu$ is unbounded from below. Since our strategy is based on find solutions to \eqref{system} as critical points of $J_\nu$, a standard approach is then to find critical points of $J_\nu$ on the underlying Nehari manifold where $J_\nu$ is bounded which allow us to minimize $\mathcal{J}_\nu$. Let us set 
\begin{equation*}
\begin{split}
\Psi(u,v)&=\left\langle \mathcal{J}'_\nu(u,v){\big|}(u,v)\right\rangle\\
&=\|(u,v)\|_\mathbb{D}^2-\int_{\mathbb{R}^N}\frac{|u|^{2^*_{s}(t)}}{|x|^{t}}dx - \int_{\mathbb{R}^N}\frac{|v|^{2^*_{s}(t)}}{|x|^{t}} dx -\nu (\alpha+\beta) \int_{\mathbb{R}^N} \frac{|u|^{\alpha} |v|^{\beta}}{|x|^t} \, dx,
\end{split}
\end{equation*}
and define the Nehari manifold associated to $\mathcal{J}_\nu$ as
\begin{equation*}
\mathcal{N}_\nu=\left\{ (u,v) \in \mathbb{D} \setminus \{(0,0)\} \, : \, \Psi(u,v) =0 \right\}.
\end{equation*}
The Nehari manifold approach is by nowadays a well-known technique. We introduce next some important properties useful in what follows. By definition, $\mathcal{N}_\nu$ contains all the non-trivial critical points of $\mathcal{J}_\nu$ in $\mathbb{D}$, i.e., given $(u,v) \in \mathcal{N}_\nu$ we have
\begin{equation} \label{Nnueq1}
 \|(u,v)\|_\mathbb{D}^2=\int_{\mathbb{R}^N}\frac{|u|^{2^*_{s}(t)}}{|x|^{t}}dx + \int_{\mathbb{R}^N}\frac{|v|^{2^*_{s}(t)}}{|x|^{t}} dx +\nu (\alpha+\beta) \int_{\mathbb{R}^N} h \frac{|u|^{\alpha} |v|^{\beta}}{|x|^t} \, dx,
\end{equation}
and, thus,
\begin{equation}\label{NN2}
\mathcal{J}_{\nu}{\big|}_{\mathcal{N}_\nu} (u,v)\!=\! \frac{2s-t}{2(N-t)} \left(\int_{\mathbb{R}^N}  \frac{|u|^{2^*_{s}(t)}}{|x|^{t}}dx+\int_{\mathbb{R}^N}\frac{|v|^{2^*_{s}(t)}}{|x|^{t}}dx\right)\!+\!\nu \left( \frac{\alpha+\beta-2}{2} \right)\!\int_{\mathbb{R}^N} h \frac{|u|^{\alpha} |v|^{\beta}}{|x|^t} dx.
\end{equation}
Moreover, because of \eqref{alphabeta} and \eqref{Nnueq1}, for any $(u,v) \in \mathcal{N}_\nu$, we have
\begin{equation*}
\begin{split}
\mathcal{J}_\nu''(u,v)[u,v]^2=&\ (2-\alpha-\beta)\|(u,v)\|_{\mathbb{D}}^2 \\
&+ (\alpha+\beta-2_{s}^*(t))\left(\int_{\mathbb{R}^N}  \frac{|u|^{2^*_{s}(t)}}{|x|^{t}}dx + \int_{\mathbb{R}^N}  \frac{|v|^{2^*_{s}(t)}}{|x|^{t}} dx\right)< 0.
\end{split}
\end{equation*}
In addition, the trivial pair $(0,0)$ is a strict minimum of $J_\nu$ on $\mathcal{N}_\nu$ since
\begin{equation*}
 \mathcal{J}_\nu''(0,0)[\varphi_1,\varphi_2]^2=\|(\varphi_1,\varphi_2)\|^2_{\mathbb{D}}>0 \quad \text{ for any } (\varphi_1,\varphi_2)\in \mathcal{N}_\nu.
\end{equation*}
Indeed, $(0,0)$ is an isolated point respect to $\displaystyle \mathcal{N}_\nu  \, \cup \, \{(0,0)\}$. As a consequence, $\mathcal{N}_\nu$ is a smooth complete manifold of codimension $1$. Also, there exists  $r_\nu>0$ such that
\begin{equation}\label{cp2}
\|(u,v)\|_{\mathbb{D}} > r_\nu\quad\text{for all } (u,v)\in \mathcal{N}_\nu.
\end{equation}
In the lines of \eqref{NN2}, observe that the functional $J_\nu$ can be also written as 
\begin{equation}\label{Nnueq}
\begin{split}
\mathcal{J}_{\nu}{\big|}_{\mathcal{N}_\nu} (u,v) =& \left( \frac{1}{2}-\frac{1}{\alpha+\beta} \right) \|(u,v)\|^2_{\mathbb{D}} \\
&+ \left( \frac{1}{\alpha+\beta}-\frac{1}{2^*_{s}(t)} \right)\left(\int_{\mathbb{R}^N} \frac{|u|^{2^*_{s}(t)}}{|x|^{t}}dx+\int_{\mathbb{R}^N} \frac{|v|^{2^*_{s}(t)}}{|x|^{t}} dx\right).
\end{split}
\end{equation}
Thus, by \eqref{cp2} and  \eqref{alphabeta} it follows that
\begin{equation*}
\mathcal{J}_{\nu} (u,v) > \left( \frac{1}{2}-\frac{1}{\alpha+\beta} \right) r^2_\nu\quad  \quad\text{for all } (u,v)\in \mathcal{N}_\nu.
\end{equation*}
Therefore, $\mathcal{J}_{\nu}$ remains bounded from below on $\mathcal{N}_\nu$. Finally, it is easy to see that
$(u,v) \in \mathbb{D}$ is a critical point of $\mathcal{J}_\nu$ if and only if $(u,v) \in \mathbb{D}$ is a critical point of $\mathcal{J}_{\nu}{\big|}_{\mathcal{N}_\nu}$. Consequently, $\mathcal{N}_\nu$ is a called a natural constraint for $J_\nu$. Hence, we can look for solutions to \eqref{system} by minimizing $J_\nu$ on the Nehari manifold $\mathcal{N}_\nu$.
\begin{definition}
We said that $(u,v)\in\mathbb{D}\setminus \{(0,0)\}$ is a non-trivial bound state for \eqref{system} if it is a non-trivial critical point of $\mathcal{J}_\nu$. 
A non-trivial and non-negative bound state $(\tilde{u},\tilde{v})$ is said to be a ground state if its energy is minimal, namely
\begin{equation}\label{cmin}
\tilde{c}_\nu\vcentcolon=\mathcal{J}_\nu(\tilde{u},\tilde{v})=\min\left\{\mathcal{J}_\nu(u,v): (u,v)\in \mathbb{D}\setminus \{(0,0)\},\; u,v\ge 0 \mbox{ and } \mathcal{J}_\nu'(u,v)=0\right\}.
\end{equation}
\end{definition}

\subsection{Semi-trivial solutions}\hfill\newline
Our analysis strongly relies on some variational properties of the semi-trivial solutions that we introduce next. Initially, we characterize the nature as a critical points of $J_\nu$ of the semi-trivial solutions $(z_\mu^{(1)},0)$ and $(0,z_\mu^{(2)})$: they become either a local minimum or a saddle point of $\mathcal{J}_{\nu}{\big|}_{\mathcal{N}_\nu}$, given certain assumptions on the coupling parameter $\nu$ and the exponents $\alpha$ and $\beta$. This classification allows us to examine the geometry of the functional $J_\nu$ and then, to obtain essential energy estimates, necessary for prove the existence of solutions. As we will see, the coupling parameter $\nu$ and the exponents $\alpha$ and $\beta$ have a subtle impact on the geometry of $J_\nu$: if $\nu$ is large enough, the minimum energy level is strictly lower than that of the semi-trivial solutions. Consequently, one can find a positive ground state solution by minimizing the functional $J_\nu$.

For any $\nu \in \mathbb{R}$, the system \eqref{system} admits two \textit{semi-trivial} positive solutions $(z^{\lambda_1,t,s},0)$ and $(0,z^{\lambda_2,t,s})$, where $z^{\lambda,t,s}$ satisfies the equation
\begin{equation}\label{eq:uncoupled}\tag{$E_{t,s}$}
 (-\Delta)^{s} z - \lambda_j \frac{z}{|x|^{2s}}=\frac{z^{2^*_{s}(t)-1}}{|x|^{t}},\qquad \text{and}\ z>0\ \text{in}\ \mathbb{R}^N\backslash\{0\}.
\end{equation}
It is easy to see that \eqref{eq:uncoupled} invariant under the scaling \[z^{\lambda_j,t,s}_\mu (x) =\mu^{-\frac{N-2s}{2}} z^{\lambda_j,t,s}\left(\frac{x}{\mu} \right).\] Because of because of \cite{Ghoussoub2018,Ghoussoub2015}, solutions to \eqref{eq:uncoupled} arise as minimizers of the Rayleigh quotient,
\begin{equation}\label{S1}
\mathcal{S}(\lambda,t,s)=\inf\limits_{\substack{u\in \mathcal{D}^{s,2}(\mathbb{R}^N)\\
u\not\equiv0}}\frac{\|u\|_{\lambda,s}^2}{\displaystyle\left(\int_{\mathbb{R}^N}\frac{|u|^{2_s^*(t)}}{|x|^t}dx\right)^{\frac{2}{2_s^*(t)}}}=\frac{\|z_{\mu}^{^{\lambda_j,t,s}}\|_{\lambda,s}^2}{\displaystyle \left(\int_{\mathbb{R}^N}\frac{|z_{\mu}^{^{\lambda_j,t,s}}|^{2_s^*(t)}}{|x|^t}dx\right)^{\frac{2}{2_s^*(t)}}}.
\end{equation}
The constant $\mathcal{S}(\lambda,t,s)$ is finite, strictly positive and achieved by the function $z^{^{\lambda,t,s}}_\mu$ solution to \eqref{eq:uncoupled} with $\lambda_j=\lambda$ (cf. \cite{Ghoussoub2018}). By direct computation we easily get
\begin{equation}\label{relation}
\|z_{\mu}^{^{\lambda,t,s}}\|_{\lambda,s}^2=\int_{\mathbb{R}^N}\frac{|z_{\mu}^{^{\lambda,t,s}}|^{2_s^*(t)}}{|x|^t}dx=[\mathcal{S}(\lambda,t,s)]^{\frac{N-t}{2s-t}}.
\end{equation}
Let us stress that, for $\lambda=0$, the constant $\mathcal{S}(0,t,s)$ is the best constant for the fractional Hardy-Sobolev inequality
\begin{equation}\label{HS_ineq}
\mathcal{S}(0,t,s) \left(\int_{\mathbb{R}^N} \frac{u^{2_s^*(t)}}{|x|^t} \, dx \right)^{\frac{2}{2_s^*(t)}}\leq \iint_{\mathbb{R}^{2N}}\frac{|u(x)-u(y)|^2}{|x-y|^{N+2s}}dydx,
\end{equation}
while, for $\lambda\in(0,\Lambda_{N,s})$,
\begin{equation}\label{HS_ineq_l}
\mathcal{S}(\lambda,t,s) \left(\int_{\mathbb{R}^N} \frac{u^{2_s^*(t)}}{|x|^t} \, dx \right)^{\frac{2}{2_s^*(t)}}\leq \iint_{\mathbb{R}^{2N}}\frac{|u(x)-u(y)|^2}{|x-y|^{N+2s}}dydx-\lambda\int_{\mathbb{R}^N}\frac{u^2}{|x|^{2s}}dx.
\end{equation}
If $\lambda=t=0$, then $\mathcal{S}(0,0,s)$ reduces to the fractional Sobolev constant which is known to be achieved by the family of Aubin--like instatons
\[u(x)=\frac{C_{N,s}}{(1+|x|^2)^{\frac{N-2s}{2}}},\]
for some constant $C_{N,s}>0$. Contrary to the case $s=1$, no explicit solution is known to \eqref{eq:uncoupled} for $0<s<1$ and $\lambda,t>0$. Actually, if $s=1$, a complete classification of solutions to \eqref{eq:uncoupled} was provided in \cite{Kang2004}: if $\lambda\in\left(0,\Lambda_{N,1}\right)$, the solutions to $(E_{t,1})$ are given by
\begin{equation}\label{Keq}
z_{\mu}^{^{\lambda_j,t,1}}(x)= \mu^{-\frac{N-2}{2}}z^{^{\lambda_j,t,1}}\left(\frac{x}{\mu}\right) \qquad \mbox{ with } \qquad z^{^{\lambda,t,1}}(x)=\dfrac{A(N,\lambda)^{\frac{N-2}{2(2-t)}}}{|x|^{a_{\lambda}}\left(1+|x|^{(2-t)(1-\frac{2}{N-2}a_{\lambda})}\right)^{\frac{N-2}{2-t}}},
\end{equation}
where $\displaystyle A(N,\lambda)=2(\Lambda_N-\lambda)\frac{N-t}{\sqrt{\Lambda_N}}$,  $a_\lambda=\sqrt{\Lambda_N}-\sqrt{\Lambda_N-\lambda}$. Moreover,
\begin{equation*}
\mathcal{S}(\lambda,t,1)=4(\Lambda_N-\lambda)\frac{N-t}{N-2}\left(\frac{N-2}{2(2-t)\sqrt{\Lambda_N-\lambda}}\frac{2\pi^{\frac{N}{2}}}{\Gamma\left(\frac{N}{2}\right)}\frac{\Gamma^2\left(\frac{N-t}{2-t}\right)}{\Gamma\left(\frac{2(N-t)}{2-t}\right)}\right)^{\frac{2-t}{N-t}}.
\end{equation*}
Observe that the constant $\mathcal{S}(\lambda,t,1)$ is decreasing in both $\lambda$ and $t$ so that the order between $\lambda_1$ and $\lambda_2$ determines the order between the energy of the associated semi-trivial solutions. Moreover
\[\lim\limits_{t\to2^-}\mathcal{S}(0,t,1)=\Lambda_{N,1}.\]
 To lighten the notation, in what follows we denote \[z^{(j)}_\mu\vcentcolon=z_{\mu}^{\lambda_j,t,s},\] and we denote by $L^{2_s^*(t),t}(\mathbb{R}^N)$ the space of measurable functions $u:\mathbb{R}^N\to\mathbb{R}$ such that
 \[ \| u\|_{L^{2_s^*(t),t}(\mathbb{R}^N)}\vcentcolon=\left(\int_{\mathbb{R}^N} \frac{|u|^{2^*_{s}(t)}}{|x|^{t}} \, dx\right)^{\frac{1}{2_s^*(t)}}<+\infty.\]
Note that the main difficulties of the work are due to the lack of compactness of the embedding $\mathcal{D}^{s,2}(\mathbb{R}^N)\hookrightarrow L^{q,t}(\mathbb{R}^N)$, with $L^{q,t}(\mathbb{R}^N)$ the space of measurable functions $u:\mathbb{R}^N\to\mathbb{R}$ such that
 \[ \| u\|_{L^{q,t}(\mathbb{R}^N)}=\left(\int_{\mathbb{R}^N} \frac{|u|^{q}}{|x|^{t}} \, dx\right)^{\frac{1}{q}}<+\infty.\]
In particular, $\mathcal{D}^{s,2}(\mathbb{R}^N)\hookrightarrow L^{q,t}(\mathbb{R}^N)$ is compact for $1\leq q< 2_s^*(t)$ and continuous up to the critical exponent $2_s^*(t)$. To continue, let us consider the decoupled energy functionals $\mathcal{J}_j:\mathcal{D}^{s,2} (\mathbb{R}^N)\mapsto\mathbb{R}$,
\begin{equation}\label{funct:Ji}
\mathcal{J}_j(u) =\frac{1}{2}\|u\|_{\lambda_j,s}^2 - \frac{1}{2^*_{s}(t)} \int_{\mathbb{R}^N} \frac{|u|^{2^*_{s}(t)}}{|x|^{t}} \, dx.
\end{equation}
Because of \cite{Ghoussoub2018,Ghoussoub2015}, the function $z_{\mu}^{(j)}$, $\mu>0$, is a global minimum of $\mathcal{J}_j$ on the corresponding Nehari manifold, namely,
\begin{equation}\label{Nj}
\mathcal{N}_j= \left\{ u \in \mathcal{D}^{s,2} (\mathbb{R}^N) \setminus \{0\} \, : \,  \left\langle \mathcal{J}'_j(u){\big|} u\right\rangle=0 \right\}.
\end{equation}
Because of \eqref{relation},
\begin{equation*}
\mathcal{J}_j(z_\mu^{(j)})=\dfrac{2s-t}{2(N-t)}\left[\mathcal{S}(\lambda_j,t,s)\right]^{\frac{N-t}{2s-t}},
\end{equation*}
so that the energy levels of the \textit{semi-trivial} solutions are given by
\begin{equation}\label{csemi}
\mathscr{C}(\lambda_1,t,s)\vcentcolon=\mathcal{J}_\nu(z_\mu^{(1)},0)=\mathcal{J}_1(z_\mu^{(1)})=\dfrac{2s-t}{2(N-t)}\left[\mathcal{S}(\lambda_1,t,s)\right]^{\frac{N-t}{2s-t}}
\end{equation}
and
\begin{equation}\label{csemi2}
\mathscr{C}(\lambda_2,t,s)\vcentcolon=\mathcal{J}_\nu(0,z_\mu^{(2)})=\mathcal{J}_2(z_\mu^{(2)})=\dfrac{2s-t}{2(N-t)}\left[\mathcal{S}(\lambda_2,t,s)\right]^{\frac{N-t}{2s-t}}.
\end{equation}
Next, in the lines of \cite[Theorem 2.2]{Abdellaoui2009}, we characterize the variational nature of the \textit{semi-trivial} couples on $\mathcal{N}_\nu$.
\begin{proposition}\label{prop_semi} Under hypotheses \eqref{alphabeta} and \eqref{H}, the following holds:
\begin{enumerate}
\item[i)] If $\alpha>2$ or $\alpha=2$ and $\nu$ small enough, then $(0,z_\mu^{(2)})$ is a local minimum of $\mathcal{J}_{\nu}{\big|}_{\mathcal{N}_\nu}$.
\item[ii)] If $\beta>2$ or $\beta=2$ and $\nu$ small enough, then $(z_\mu^{(1)},0)$ is a local minimum of $\mathcal{J}_{\nu}{\big|}_{\mathcal{N}_\nu}$.
\item[iii)] If $\alpha<2$ or $\alpha=2$ and $\nu$ large enough, then $(0,z_\mu^{(2)})$ is a saddle point for $\mathcal{J}_{\nu}{\big|}_{\mathcal{N}_\nu}$.
\item[iv)] If $\beta<2$ or $\beta=2$ and $\nu$ large enough, then $(z_\mu^{(1)},0)$ is a saddle point for $\mathcal{J}_{\nu}{\big|}_{\mathcal{N}_\nu}$.
\end{enumerate}
\end{proposition}
\begin{proof}
Let us prove $i)$. Take $(\varphi,z_\mu^{2}+\psi)\in\mathcal{N}_\nu$ with $\mu>0$ and $\alpha>2$ and note that, by \eqref{Nnueq1}, 
\begin{equation}\label{Nnueq1bis}
\begin{split}
\|(\varphi,z_\mu^{2}+\psi)\|^2_{\mathbb{D}}=&\int_{\mathbb{R}^N} \frac{|\varphi|^{2^*_{s}(t)}}{|x|^{t}}dx\! + \! \int_{\mathbb{R}^N}\frac{|z_\mu^{(2)}+\psi|^{2^*_{s}(t)}}{|x|^{t}} dx\! + \! (\alpha+\beta)\nu \int_{\mathbb{R}^N} h \frac{|\varphi|^{\alpha} |z_\mu^{(2)}+\psi|^{\beta}}{|x|^t} dx.
\end{split}
\end{equation}
Next, take $\overline{t}$ such that $\overline{t}(z_\mu^{(2)}+\psi)\in \mathcal{N}_{2}$, with $\mathcal{N}_{2}$ as in \eqref{Nj}. Actually, by using the definition of $\mathcal{N}_2$ and \eqref{Nnueq1bis}, the value of $t$ is determined by  the following expression (see \eqref{normH}),
{\tiny
\begin{equation}\label{tminmax}
\begin{split}
 \overline{t}&=\left(\dfrac{\|z_\mu^{(2)}+\psi\|^2_{\lambda_2,s}}{\displaystyle{\int_{\mathbb{R}^N}\frac{|z_\mu^{(2)}+\psi|^{2^*_{s}(t)}}{|x|^{t}}}dx} \right)^{\frac{1}{2^*_{s}(t)-2}}= \left(1-\dfrac{\|\varphi\|^2_{\lambda_1}-\displaystyle\int_{\mathbb{R}^N} \frac{|\varphi|^{2^*_{s}(t)}}{|x|^{t}} \, dx - (\alpha+\beta)\nu \int_{\mathbb{R}^N} h(x) \frac{|\varphi|^{\alpha} |z_\mu^{(2)}+\psi|^{\beta}}{|x|^t} dx}{\displaystyle\int_{\mathbb{R}^N}\dfrac{|z_\mu^{(2)}+\psi|^{2^*_{s}(t)}}{|x|^{t}}} \right)^{\frac{1}{2^*_{s}(t)-2}}.
\end{split}
\end{equation}
}
Because of the hypothesis \eqref{H} on the function $h$, { using twice the Holder's inequality with exponents $p=\dfrac{2_{s}^*(t)}{2_{s}^*(t)-\alpha-\beta}$ and  $q=\dfrac{2_{s}^*(t)}{\alpha+\beta}$ and $\tilde p = \dfrac{2_{s}^*(t)}{\alpha}$ and $\tilde{q} = \dfrac{2_{s}^*(t)}{\beta}$ respectively, we have
\begin{equation}\label{Holder1}
\begin{split}
 \int_{\mathbb{R}^N} h(x) \frac{|\varphi|^{\alpha} |z_\mu^{(2)}+\psi|^{\beta}}{|x|^t}dx & \leq \left(\int_{\mathbb{R}^N} \dfrac{h^{\frac{2_{s}(t)^*}{2_{s}^*(t)-\alpha-\beta}}}{|x|^{t}} \right)^{\frac{2_{s}^*(t)-\alpha-\beta}{2_{s}^*(t)}} \left(\int_{\mathbb{R}^N} \dfrac{(|\varphi|^\alpha |z_\mu^{(2)}+\psi|^\beta)^{\frac{2_{s}^*(t)}{\alpha+\beta}}}{|x|^t} \right)^\frac{\alpha+\beta}{2_{s}^*(t)} \\ 
&\leq C(h)  \left( \int_{\mathbb{R}^N} \frac{|\varphi|^{2^*_{s}(t)}}{|x|^{t}} dx\right)^{\frac{\alpha}{2^*_{s}(t)}}  \left( \int_{\mathbb{R}^N} \frac{|z_\mu^{(2)}+\psi|^{2^*_{s}(t)}}{|x|^{t}} dx\right)^{\frac{\beta}{2^*_{s}(t)}},
\end{split}
\end{equation}
}
By combining \eqref{tminmax} and \eqref{Holder1}, we infer
\begin{equation}\label{texpan1}
\overline{t}^{2}= 1-\frac{2}{2^*_{s}(t)-2}\frac{\|\varphi \|^{2}_{\lambda_1,s}(1+o(1))}{\displaystyle\int_{\mathbb{R}^N}\dfrac{|z_\mu^{(2)}+\psi|^{2^*_{s}(t)}}{|x|^{t}}dx}, \qquad \mbox{ as } \|(\varphi,\psi)\|_{\mathbb{D}}\to 0.
\end{equation}
\begin{equation}\label{texpan2}
\overline{t}^{2^*_{s}(t)}= 1-\frac{2^*_{s}(t)}{2^*_{s}(t)-2}\frac{\|\varphi \|^{2}_{\lambda_1,s}(1+o(1))}{\displaystyle\int_{\mathbb{R}^N}\dfrac{|z_\mu^{(2)}+\psi|^{2^*_{s}(t)}}{|x|^{t}}dx}, \qquad \mbox{ as } \|(\varphi,\psi)\|_{\mathbb{D}}\to 0.
\end{equation}
As the decoupled energy functional $\mathcal{J}_2$ achieves its minimum in $z_\mu^{(2)}$, then
\begin{equation*}
\mathcal{J}_2(\overline{t}(z_\mu^{(2)}+\psi))-\mathcal{J}_2(z_\mu^{(2)})=\mathcal{J}_\nu(0,\overline{t}(z_\mu^{(2)}+\psi))-\mathcal{J}_\nu(0,z_\mu^{(2)}) \ge 0.
\end{equation*}
Next, comparing the energy of $(\varphi,z_\mu^{(2)}+\psi)$ and $(0,\overline{t}(z_\mu^{(2)}+\psi))$, from \eqref{Nnueq1bis}, \eqref{tminmax}, \eqref{texpan1} and \eqref{texpan2}, we obtain
\begin{equation}\label{compar}
\begin{split}
 \mathcal{J}_\nu(\varphi,z_\mu^{(2)}+\psi)\!-\!\mathcal{J}_\nu(0,\overline{t}(z_\mu^{(2)}+\psi))=&\frac{1}{2} \| \varphi\|_{\lambda_1,s}^2+\frac{1}{2}(1-\overline{t}^2)\|z_\mu^{(2)}+\psi\|_{\lambda_2,s}^2\\
 &-\frac{1}{2_{s}^*(t)}\int_{\mathbb{R}^N}\dfrac{| \varphi|^{2^*_{s}(t)}}{|x|^{t}}dx\\
 &+\frac{1}{2_{s}^*(t)}(1-\overline{t}^{2^*_{s}(t)}) \int_{\mathbb{R}^N}\dfrac{|z_\mu^{(2)}+\psi|^{2^*_{s}(t)}}{|x|^{t}}dx\\
& - \nu \int_{\mathbb{R}^N} h \frac{|\varphi|^{\alpha} |z_\mu^{(2)}+\psi|^{\beta}}{|x|^t}dx \\
=& \frac{1}{2} (1+o(1)) \| \varphi\|_{\lambda_1,s}^2,\qquad \mbox{ as } \|(\varphi,\psi)\|_{\mathbb{D}}\to 0.
\end{split}
\end{equation}
As a consequence, $\mathcal{J}_\nu(\varphi,z_\mu^{(2)}+\psi)-\mathcal{J}_\nu(0,z_\mu^{(2)}))\ge 0.$
Taking the perturbation $(\varphi,z_\mu^{(2)}+\psi)$ small enough in the $\mathbb{D}$-norm sense, we conclude that $(0,z_\mu^{(2)})$ is a local minimum of $\mathcal{J}_\nu$ on $\mathcal{N}_\nu$. In the case $\alpha=2$, from \eqref{compar} and \eqref{Holder1}, we get
\begin{equation*}
\begin{split}
\mathcal{J}_\nu(\varphi,z_\mu^{(2)}+\psi)-\mathcal{J}_\nu(0,\overline{t}(z_\mu^{(2)}+\psi)) &= \frac{1}{2} (1+o(1)) \| \varphi\|_{\lambda_1,s}^2 - \nu \int_{\mathbb{R}^N} h \frac{|\varphi|^{2} |z_\mu^{(2)}+\psi|^{\beta}}{|x|^t}dx \\
&= \left( \frac{1}{2}-\nu C'+o(1) \right) \, \| \varphi\|_{\lambda_1,s}^2,\qquad \mbox{ as } \|(\varphi,\psi)\|_{\mathbb{D}}\to 0.
\end{split}
\end{equation*}
Then for $\nu$ sufficiently small, we conclude that $(0,z_\mu^{(2)})$ is a local minimum of $\mathcal{J}_{\nu}$ in ${\mathcal{N}_\nu}$. Item $ii)$ follows analogously.

Next, we prove $iii)$. Take $f(\tau)$ such that $\displaystyle
(f(\tau)\tau\varphi,f(\tau)z_\mu^{(2)})\in\mathcal{N}_\nu$, with $\varphi\in \mathcal{D}^{s,2}(\mathbb{R}^N)\setminus\{0\}$ and $\mu>0$. Hence, due to \eqref{Nnueq1},
\begin{equation*}
\begin{split}
\tau^2\|\varphi\|^2_{\lambda_1}+\|z_\mu^{(2)}\|^2_{\lambda_2}
 =&[f(\tau)]^{2^*_{s}(t)-2}|\tau|^{2^*_{s}(t)}\left( \int_{\mathbb{R}^N}\dfrac{| \varphi|^{2^*_{s}(t)}}{|x|^{t}}dx+\int_{\mathbb{R}^N}\dfrac{| z_\mu^{(2)}|^{2^*_{s}(t)}}{|x|^{t}}dx \right)\\
& + (\alpha+\beta)\nu[f(\tau)]^{\alpha+\beta-2}|\tau|^{\alpha}\int_{\mathbb{R}^N} h(x) \frac{|\varphi|^{\alpha} |z_\mu^{(2)}|^{\beta}}{|x|^t}dx.
\end{split}
\end{equation*}
Observe that $f(0)=1$. In addition, because of the Implicit Function Theorem, we have $f \in C^1(\mathbb{R})$ and
{\small
\begin{equation*}
f'(\tau)=\frac{\displaystyle 2\tau\|\varphi\|^2_{\lambda_1,s} - 2^*_{s}(t)f^{2^*_{s}(t)-2}|\tau|^{2^*_{s}(t)-2}\tau \int_{\mathbb{R}^N}\dfrac{| \varphi|^{2^*_{s}(t)}}{|x|^{t}}dx - \alpha (\alpha+\beta)\nu f^{\alpha+\beta-2}|\tau|^{\alpha-2}\tau\int_{\mathbb{R}^N} h \frac{|\varphi|^{\alpha} |z_\mu^{(2)}|^{\beta}}{|x|^t}dx }{\displaystyle(2^*_{s}(t)-2)f^{2^*_{s}(t)-3}|t|^{2^*_{s}(t)}\left(\int_{\mathbb{R}^N}\dfrac{| \varphi|^{2^*_{s}(t)}}{|x|^{t}}dx + \int_{\mathbb{R}^N}\dfrac{| z_\mu^{(2)}|^{2^*_{s}(t)}}{|x|^{t}}dx\right) + \nu \delta f^{\alpha+\beta-3}|\tau|^{\alpha}\int_{\mathbb{R}^N} h \frac{|\varphi|^{\alpha} |z_\mu^{(2)}|^{\beta}}{|x|^t}dx},
\end{equation*}
}
where $\delta=(\alpha+\beta)(\alpha+\beta-2)$. Since $\alpha<2$, we can write the previous expression as 
\begin{equation*}
f'(\tau)=\dfrac{\displaystyle- \alpha (\alpha+\beta)\nu\int_{\mathbb{R}^N} h(x) \frac{|\varphi|^{\alpha} |z_\mu^{(2)}|^{\beta}}{|x|^t}dx }{\displaystyle (2^*_{s}(t)-2) \int_{\mathbb{R}^N}\dfrac{| z_\mu^{(2)}|^{2^*_{s}(t)}}{|x|^{t}}dx }\, |\tau|^{\alpha-2}\tau\, \, (1+o(1)), \qquad \mbox{ as }  \tau\to 0.
\end{equation*}
By integrating and using that $f(0)=1$, we obtain
\begin{equation*}
f(\tau)= 1- \frac{\displaystyle (\alpha+\beta)\nu\int_{\mathbb{R}^N} h(x) \frac{|\varphi|^{\alpha} |z_\mu^{(2)}|^{\beta}}{|x|^t}dx }{\displaystyle (2^*_{s}(t)-2) \int_{\mathbb{R}^N}\frac{| z_\mu^{(2)}|^{2^*_{s}(t)}}{|x|^{t}} dx}\, |\tau|^{\alpha} \, \, (1+o(1)), \qquad \mbox{ as }  \tau\to 0,
\end{equation*}
and consequently, by Taylor expansion,
\begin{equation}\label{efe2star}
f^{2^*_{s}(t)}(\tau)= 1- \frac{\displaystyle (N-t)(\alpha+\beta)\nu\int_{\mathbb{R}^N} h(x)\frac{|\varphi|^{\alpha} |z_\mu^{(2)}|^{\beta}}{|x|^t} dx}{\displaystyle (2-t) \int_{\mathbb{R}^N}\dfrac{| z_\mu^{(2)}|^{2^*_{s}(t)}}{|x|^{t}} dx}\, |\tau|^{\alpha} \, \, (1+o(1)), \qquad \mbox{ as }  \tau\to 0.
\end{equation}
To conclude, we note that, by \eqref{NN2} and \eqref{efe2star},
\begin{equation*}\label{compar2}
\begin{split}
\mathcal{J}_\nu((f(\tau)&\tau\varphi,f(\tau)z_\mu^{(2)})-\mathcal{J}_\nu(0,z_\mu^{(2)})\\
&=\frac{2-t}{2(N-t)}\left([f(\tau)]^{2^*_{s}(t)}\int_{\mathbb{R}^N}\dfrac{| \varphi|^{2^*_{s}(t)}}{|x|^{t}}dx + ([f(\tau)]^{2^*_{s}(t)}-1)  \int_{\mathbb{R}^N}\dfrac{|z_\mu^{(2)}|^{2^*_{s}(t)}}{|x|^{t}}dx\right)\\
&\ + \left(\frac{\alpha+\beta-2}{2}\right)\nu |\tau|^{\alpha} \int_{\mathbb{R}^N} h \frac{|\varphi|^{\alpha} |z_\mu^{(2)}|^{\beta}}{|x|^t}dx \\
&=- \left(\frac{\alpha+\beta}{2}\right)\nu |\tau|^{\alpha}\!\! \int_{\mathbb{R}^N} h \frac{|\varphi|^{\alpha} |z_\mu^{(2)}|^{\beta}}{|x|^t}dx + \left(\frac{\alpha+\beta-2}{2}\right)\nu |\tau|^{\alpha}\!\! \int_{\mathbb{R}^N} h \frac{|\varphi|^{\alpha} |z_\mu^{(2)}|^{\beta}}{|x|^t}dx+o(|\tau|^\alpha) \\
&=- \nu |\tau|^{\alpha} \int_{\mathbb{R}^N} h \frac{|\varphi|^{\alpha} |z_\mu^{(2)}|^{\beta}}{|x|^t}dx+o(|\tau|^\alpha)\quad\text{as }\tau\to0.
\end{split}
\end{equation*}
Thus, $(0,z_\mu^{(2)})$ is a local minimum of $\mathcal{J}_\nu$ along a path lying on the Nehari manifold $\mathcal{N}_\nu$ since,  for $\tau$ small enough,
\[\mathcal{J}_\nu((f(\tau)\tau\varphi,f(\tau)z_\mu^{(2)})-\mathcal{J}_\nu(0,z_\mu^{(2)}))<0.\]
Moreover, for any pair $(0,\psi) \in \mathcal{N}_\nu$, we have $\mathcal{J}_\nu(0,\psi)>\mathcal{J}_\nu(0,z_\mu^{(2)})$, as $z_\mu^{(2)}$ is an absolute minimum for $\mathcal{J}_2$. Then, $(0,z_\mu^{(2)})$ is a local minimum in $\{0\}\times \mathcal{N}_2\subset\mathcal{N}_{\nu}$. Hence, if $\alpha<2$, the couple $(0,z_\mu^{(2)})$ is a saddle point of $\mathcal{J}_\nu$ in $\mathcal{N}_\nu$. In case of having $\alpha=2$, taking $((f(\tau)\tau\varphi,f(\tau)z_\mu^{(2)}) \in \mathcal{N}_\nu$, we get
\begin{equation*}
f'(\tau)=\frac{\displaystyle 2\|\varphi\|^2_{\lambda_1,s}- 2 (2+\beta)\nu\int_{\mathbb{R}^N} h \frac{|\varphi|^{2} |z_\mu^{(2)}|^{\beta}}{|x|^t}dx}{\displaystyle (2^*_{s}(t)-2) \int_{\mathbb{R}^N}\dfrac{| z_\mu^{(2)}|^{2^*_{s}(t)}}{|x|^{t}}dx }\, \tau \, \, (1+o(1)), \qquad \mbox{ as }  \tau\to 0,
\end{equation*}
so that
\begin{equation*}
f(\tau)=1-\frac{\displaystyle (2+\beta)\nu\int_{\mathbb{R}^N} h \frac{|\varphi|^{2} |z_\mu^{(2)}|^{\beta}}{|x|^t}dx -\|\varphi\|^2_{\lambda_1,s} }{\displaystyle (2^*_{s}(t)-2) \int_{\mathbb{R}^N}\dfrac{| z_\mu^{(2)}|^{2^*_{s}(t)}}{|x|^{t}} dx}\, \tau^2 \, \, (1+o(1)), \qquad \mbox{ as }  \tau\to 0,
\end{equation*}
and, thus,
\begin{equation*}
\mathcal{J}_\nu((f(\tau)\tau\varphi,f(\tau)z_\mu^{(2)})-\mathcal{J}_\nu(0,z_\mu^{(2)}) = \left(\frac{1}{2} \|\varphi\|_{\lambda_1,s}^2 - \nu \int_{\mathbb{R}^N} h \frac{|\varphi|^{2} |z_\mu^{(2)}|^{\beta}}{|x|^t}dx \right)\tau^2 +o(\tau^2).
\end{equation*}
Subsequently, if $\nu$ is large enough the previous quantity will be negative and $(0,z_\mu^{(2)})$ is a saddle point of $\mathcal{J}_\nu$ in $\mathcal{N}_\nu$. Therefore, reasoning as before, the thesis $iii)$ is concluded. Item $iv)$ follows similarly
\end{proof}
\begin{remark}
In case of having $\alpha+\beta=2_s^*(t)$, the constant $C(h)=\|h\|_{L^{\infty}(\mathbb{R}^N)}$ in \eqref{Holder1}. 
\end{remark}
When looking for positive solutions to system \eqref{system} it is useful to analyze the associated truncated system, namely 
\begin{equation}\label{systemp}
\left\{\begin{array}{ll}
\displaystyle (-\Delta)^s u - \lambda_1 \frac{u}{|x|^{2s}}-\frac{(u^+)^{2_{s}^*(t)-1}}{|x|^{t}}=  \nu\alpha h(x) \frac{(u^+)^{\alpha-1}\, (v^+)^{\beta}}{|x|^{t}}  &\text{in }\mathbb{R}^N,\vspace{.3cm}\\
\displaystyle (-\Delta)^s v - \lambda_2 \frac{v}{|x|^{2s}}-\frac{(v^+)^{2_{s}^*(t)-1}}{|x|^{t}}= \nu \beta h(x)  \frac{(u^+)^{\alpha}\, (v^+)^{\beta-1}}{|x|^{t}} &\text{in }\mathbb{R}^N,
\end{array}\right.
\end{equation}
where  $u^+=\max\{u,0\}$. Denoting by $u^-=\min\{u,0\}$ the negative part of the function $u$, it follows that $u=u^+ + u^-$. Moreover, by definition,  $supp(u^+)\cap supp(u^-)=\emptyset$. The system \eqref{systemp} is a variational system whose solutions correspond to critical points of
\begin{equation*}
\mathcal{J}^+_\nu (u,v)= \|(u,v)\|^2_{\mathbb{D}}- \frac{1}{2^*_{s}(t)} \int_{\mathbb{R}^N} \frac{(u^+)^{2^*_{s}(t)}}{|x|^{t}} dx - \frac{1}{2^*_{s}(t)} \int_{\mathbb{R}^N} \frac{(v^+)^{2^*_{s}(t)}}{|x|^{t}} dx -\nu \int_{\mathbb{R}^N} h \frac{(u^+)^\alpha (v^+)^{\beta}}{|x|^{t}} dx,
\end{equation*}
defined in $\mathbb{D}$. We denote by $\mathcal{N}^+_\nu$ the Nehari manifold associated to $\mathcal{J}^+_\nu $. Let us stress that, in the lines of \eqref{Nnueq1}, we have
\begin{equation} \label{N1}
\|(u,v)\|_\mathbb{D}^2=\int_{\mathbb{R}^N} \frac{(u^+)^{2^*_{s}(t)}}{|x|^{t}} dx + \int_{\mathbb{R}^N} \frac{(v^+)^{2^*_{s}(t)}}{|x|^{t}} dx +\nu (\alpha+\beta) \int_{\mathbb{R}^N} h(x) \frac{(u^+)^{\alpha} (v^+)^{\beta}}{|x|^t}  dx,
\end{equation}
for $(u,v) \in \mathcal{N}^+_\nu$. Moreover,
\begin{equation}\label{N2}
\begin{split}
\mathcal{J}^+_{\nu}{\big|}_{\mathcal{N}_\nu} (u,v) =& \left( \frac{1}{2}-\frac{1}{\alpha+\beta} \right) \|(u,v)\|^2_{\mathbb{D}}\\
& + \left( \frac{1}{\alpha+\beta}-\frac{1}{2^*_{s}(t)} \right)\left(\int_{\mathbb{R}^N} \frac{(u^+)^{2^*_{s}(t)}}{|x|^{t}}dx+ \int_{\mathbb{R}^N} \frac{(v^+)^{2^*_{s}(t)}}{|x|^{t}} dx\right).
\end{split}
\end{equation}
We conclude this section by establishing the fractional Hardy-Sobolev version of \cite[Lemma 3.3]{Abdellaoui2009} or the Hardy-Sobolev version of \cite[Lemma 3.3]{Kumar2022}. Its proof follows analogously so we omit the details. 
\begin{lemma}\label{lem_algebr}
Let $A, B>0$, $0\leq s<2$ and $\theta \ge 2$ and consider the set
\begin{equation*}
\Sigma_\nu=\{\sigma \in (0,+\infty)  \, : \, A \sigma^{\frac{2}{2_{s}^*(t)}} < \sigma + B \nu \sigma^{\frac{\theta}{2_s^*(t)}} \}.
\end{equation*}
Then, for every $\varepsilon>0$ there exists $\tilde{\nu}>0$ such that
\begin{equation*}
\inf_{\Sigma_\nu} \sigma > (1-\varepsilon) A^{\frac{N-t}{2s-t}} \qquad \mbox{ for any } 0<\nu<\tilde{\nu}.
\end{equation*}
\end{lemma}


\section{The Palais-Smale condition}\label{sect:Palais-Smale}
Due to the lack of compactness at the critical exponent and the consequent concentration phenomena, a careful compactness analysis for Palais-Smale sequences is performed along this section. 
\begin{definition}
Let $V$ be a Banach space. We say that $\{u_n\} \subset V$ is a PS sequence at level $c$ for an energy functional $\mathscr{F}:V\mapsto\mathbb{R}$ if
\begin{equation*}
\mathscr{F}(u_n) \to c \quad \hbox{ and }\quad  \mathscr{F}'(u_n) \to 0\quad\mbox{in}\ V'\quad \hbox{as}\quad n\to + \infty,
\end{equation*}
where $V'$ is the dual space of $V$. Moreover, we say that the functional $\mathscr{F}$ satisfies the PS condition at level $c$ if every PS sequence at $c$ for $\mathscr{F}$ has a strongly convergent subsequence.
\end{definition}
Next we state some natural properties useful in the sequel. Their proofs are somehow standard and similar to \cite[Lemma 3.2]{Colorado2022} and \cite[Lemma 3.3]{Colorado2022} (for the local case) or \cite[Lemma 3.1]{Kumar2022} and \cite[Lemma 3.1]{Kumar2022} (for the fractional case) respectively, so we omit the details.
\begin{lemma}\label{lem_PSNehari}
Let $\{(u_n,v_n)\} \subset \mathcal{N}_\nu$ be a PS sequence for $\mathcal{J}_\nu {\big|}_{\mathcal{N}_\nu}$ at level $c\in\mathbb{R}$. Then, $\{(u_n,v_n)\}$ is a PS sequence for $\mathcal{J}_\nu$ in $\mathbb{D}$, namely
\begin{equation*}
\mathcal{J}_{\nu}'(u_n,v_n)\to0 \quad \mbox{ in } \mathbb{D}' \quad\text{as }n\to+\infty.
\end{equation*}
\end{lemma}
\begin{lemma}\label{lem_PSBounded}
Let $\{(u_n,v_n)\} \subset \mathbb{D}$ be a PS sequence for $\mathcal{J}_\nu$ at level $c\in\mathbb{R}$. Then,  $\|(u_n,v_n)\|_{\mathbb{D}}<C$.
\end{lemma}

\subsection{Subcritical range $ \alpha+\beta < 2_{s}^*$}\hfill\newline
In the following, we prove the Palais-Smale compactness condition under a quantization of the energy levels of $\mathcal{J}_\nu$. The proof is based on a concentration-compactness argument relying on \cite{Bonder2018, Chen2018}.  Observe that due to the presence of the Hardy-Sobolev terms only concentration at the origin or at infinity can take place as, far from those points, all the involved terms are of subcritical nature.
\begin{lemma}\label{lemmaPS2}
 Assume $\alpha+\beta<2_{s}^*(t)$.  Then, the functional $\mathcal{J}_\nu$ satisfies the PS condition for any level $c$ such that
\begin{equation}\label{hyplemmaPS2}
c<\min\left\{ \mathscr{C}(\lambda_1,t,s), \mathscr{C}(\lambda_2,t,s) \right\}.
\end{equation}
\end{lemma}

\begin{proof}
Since, by Lemma~\ref{lem_PSBounded}, any PS sequence is bounded in $\mathbb{D}$, there exists $(\tilde{u},\tilde{v})\in\mathbb{D}$ such that, up to a subsequence,
 \begin{align*}
(u_n,v_n) \rightharpoonup (\tilde{u},\tilde{v})& \quad \hbox{weakly in  } \mathbb{D},\\
(u_n,v_n) \to (\tilde{u},\tilde{v})&\quad \hbox{strongly in  } L_{loc}^{q,t}(\mathbb{R}^N)\times L_{loc}^{q,t}(\mathbb{R}^N)\text{ for } 1\leq q<2_s^*(t),\\
(u_n,v_n) \to (\tilde{u},\tilde{v})&\quad \hbox{a.e. in  }\mathbb{R}^N.
\end{align*}
Following \cite{Bonder2018}, let us define
\begin{equation*}
|D^su|^2=\int_{\mathbb{R}^N}\frac{|u(x+h)-u(x)|^2}{|h|^{N+2s}}dh.
\end{equation*}
Note that the term $|D^su|^2$ is the fractional counterpart to the classical (and local) term $|\nabla u|^2$. Because of \cite[Theorem 1.1]{Bonder2018}, \cite[Lemma 4.5]{Chen2018}, there exist a subsequence $\{(u_n,v_n)\}$ and positive numbers $\mu_0$, $\rho_0$, $\eta_0$, $\overline{\mu}_0$, $\overline{\rho}_0$ and $\overline{\eta}_0$ such that, in the sense of measures,
\begin{equation}\label{con-comp}
\left\{
\begin{array}{l}
|D^s u_n|^2 \rightharpoonup d\mu\ge |D^s \tilde{u}|^2+\mu_0\delta_0,\qquad |D^s v_n|^2\rightharpoonup d\overline{\mu}\ge |D^s \tilde{v}|^2+\overline{\mu}_0\delta_0,\\
\\
\dfrac{|u_n|^{2_{s}^*(t)}}{|x|^{t}} \rightharpoonup d\rho=\dfrac{|\tilde{u}|^{2_{s}^*(t)}}{|x|^{t}}+\rho_0\delta_0,\qquad\ \ 
\dfrac{|v_n|^{2_{s}^*(t)}}{|x|^{t}}  \rightharpoonup d\overline{\rho}= \dfrac{|\tilde{v}|^{2_{s}^*(t)}}{|x|^{t}}+\overline{\rho}_0\delta_0,\\
\\
\dfrac{u_n^2}{|x|^{2s}} \rightharpoonup d\eta=\dfrac{\tilde{u}^2}{|x|^{2s}}+\eta_0\delta_0,\qquad \mkern+40mu 
\dfrac{v_n^2}{|x|^{2s}} \rightharpoonup d\overline{\eta}=\dfrac{\tilde{v}^2}{|x|^{2s}}+\overline{\eta}_0\delta_0.
\end{array}
\right.
\end{equation}
We also set the numbers
\begin{equation}\label{con-compinfty}
\begin{split}
\mu_{\infty}&=\lim\limits_{R\to+\infty}\limsup\limits_{n\to+\infty}\int_{|x|>R}|D^s u_n|^{2}dx,\\
\rho_{\infty}&=\lim\limits_{R\to+\infty}\limsup\limits_{n\to+\infty}\int_{|x|>R}\frac{|u_n|^{2_{s}^*(t)}}{|x|^{t}}dx,\\
\eta_{\infty}&=\lim\limits_{R\to+\infty}\limsup\limits_{n\to+\infty}\int_{|x|>R}\frac{u_n^{2}}{|x|^{2s}}dx.
\end{split}
\end{equation}
to encode the concentration at infinity of the sequence $\{u_n\}$. The concentration at infinity of $\{v_n\}$ is encoded analogously by the numbers $\overline{\mu}_{\infty}$, $\overline{\rho}_{\infty}$ and $\overline{\eta}_{\infty}$. Next, let $B_r(0)$ be the ball of radius $r>0$ centered at $0$ and let $\varphi_{\varepsilon}(x)$ be a smooth cut-off function centered at $0$, i.e., $\varphi_{\varepsilon}\in C^{\infty}(\mathbb{R}^+_0)$,
\begin{equation}\label{cutoff}
\varphi_{\varepsilon}=1 \quad \hbox{in}\quad B_{\frac{\varepsilon}{2}}(0),\quad \varphi_{\varepsilon}=0 \quad \hbox{in}
\quad B_{\varepsilon}^c(0)\quad \hbox{and}\quad\displaystyle|\nabla \varphi_{\varepsilon}|\leq \frac{4}{\varepsilon}.
\end{equation}
Testing $\mathcal{J}_{\nu}'(u_n,v_n)$ with $(u_n\varphi_{\varepsilon},0)$ we get

\begin{equation}\label{eq:cc1}
\begin{split}
0=&\lim\limits_{n\to+\infty}\left\langle \mathcal{J}'_\nu(u_n,v_n){\big|}(u_n\varphi_{\varepsilon},0)\right\rangle\\
=&\lim\limits_{n\to+\infty}\left(\iint_{\mathbb{R}^{2N}}\frac{|u_n(x)-u_n(y)|^2}{|x-y|^{N+2s}}\varphi_{\varepsilon}dydx\right.\\
&\mkern+65mu+\iint_{\mathbb{R}^{2N}}\frac{(u_n(x)-u_n(y))(\varphi_{\varepsilon}(x)-\varphi_{\varepsilon}(y))}{|x-y|^{N+2s}}u_n(y)dydx\\
&\mkern+65mu-\left.\lambda_1\int_{\mathbb{R}^N}\frac{u_n^2}{|x|^{2s}}\varphi_{\varepsilon}dx-\int_{\mathbb{R}^N}\frac{|u_n|^{2_{s}^*(t)}}{|x|^{t}}\varphi_{\varepsilon}dx-\nu\alpha\int_{\mathbb{R}^N}h\frac{|u_n|^\alpha  |v_n|^\beta}{|x|^{t}}  \varphi_{\varepsilon} dx\right)\\
=&\int_{\mathbb{R}^N}\varphi_{\varepsilon}d\mu-\lambda_1\int_{\mathbb{R}^N}\varphi_{\varepsilon}d\eta-\int_{\mathbb{R}^N}\varphi_{\varepsilon}d\rho\\
&+\lim\limits_{n\to+\infty}\iint_{\mathbb{R}^{2N}}\frac{(u_n(x)-u_n(y))(\varphi_{\varepsilon}(x)-\varphi_{\varepsilon}(y))}{|x-y|^{N+2s}}u_n(y)dydx\\
&-\nu\alpha \lim\limits_{n\to+\infty}\int_{\mathbb{R}^N}h\frac{|u_n|^\alpha  |v_n|^\beta}{|x|^{t}}  \varphi_{\varepsilon} dx.
\end{split}
\end{equation}
Next we show that
\begin{equation*}
\lim\limits_{\varepsilon\to0^+}\lim\limits_{n\to+\infty}\iint_{\mathbb{R}^{2N}}\frac{(u_n(x)-u_n(y))(\varphi_{\varepsilon}(x)-\varphi_{\varepsilon}(y))}{|x-y|^{N+2s}}u_n(y)dydx=0.
\end{equation*}
Actually, since $\{u_n\varphi_{\varepsilon}\}$ is bounded, by using H\"older inequality, 
\begin{equation*}
\iint_{\mathbb{R}^{2N}}\frac{(u_n(x)-u_n(y))(\varphi_{\varepsilon}(x)-\varphi_{\varepsilon}(y))}{|x-y|^{N+2s}}u_n(y)dydx\leq C\left(\iint_{\mathbb{R}^{2N}}\frac{|\varphi_{\varepsilon}(x)-\varphi_{\varepsilon}(y)|^2|u_n(y)|^2}{|x-y|^{N+2s}}dydx\right)^{\frac{1}{2}},
\end{equation*}
and, by \cite[Lemma 2.2, Lemma 2.4]{Bonder2018} and \cite[Lemma 2.3]{Xiang2022},
\begin{equation*}
\lim\limits_{\varepsilon\to0^+}\lim\limits_{n\to+\infty}\iint_{\mathbb{R}^{2N}}\frac{|\varphi_{\varepsilon}(x)-\varphi_{\varepsilon}(y)|^2|u_n(y)|^2}{|x-y|^{N+2s}}dydx=0.
\end{equation*}
On the other hand, since $\alpha+\beta< 2_s^*(t)$, we have
\begin{equation*}
\begin{split}
\int_{\mathbb{R}^N}h\frac{|u_n|^\alpha  |v_n|^\beta}{|x|^{t}}  \varphi_{\varepsilon} dx\leq&
\left(\int_{\mathbb{R}^N}h\frac{\varphi_{\varepsilon} }{|x|^{t}}  dx\right)^{1-\frac{\alpha+\beta}{2_s^*(t)}}
\left(\int_{\mathbb{R}^N}h\varphi_{\varepsilon}\frac{ |u|^{2_s^*(t)}}{|x|^{t}}  dx\right)^{\frac{\alpha}{2_s^*(t)}}
\left(\int_{\mathbb{R}^N}h\varphi_{\varepsilon}\frac{ |v|^{2_s^*(t)}}{|x|^{t}}  dx\right)^{\frac{\beta}{2_s^*(t)}},
\end{split}
\end{equation*}
and, moreover,
\begin{equation*}
\lim\limits_{\varepsilon\to0^+}\int_{\mathbb{R}^N}h(x)\frac{\varphi_{\varepsilon} }{|x|^{t}}  dx=0.
\end{equation*}
Thus, from \eqref{eq:cc1}, $\mu_0-\lambda_1\eta_0-\rho_0\leq 0$. Since, by \eqref{HS_ineq_l}, we also have
\begin{equation}\label{ineq:cca}
\mu_0-\lambda_1\eta_0\ge \mathcal{S}(\lambda_1,t,s)\rho_0^{\frac{2}{2_{s}^*(t)}},
\end{equation}
from \eqref{hardy_inequality} and \eqref{HS_ineq}, we conclude
\begin{equation}\label{ineq:cc0a}
\rho_0=0\quad\text{or}\quad \rho_0\ge \left[\mathcal{S}(\lambda_1,t,s)\right]^{\frac{N-t}{2s-t}}.
\end{equation}
Similar relations hold for the concentration coefficients $\overline{\mu}_0$, $\overline{\eta}_0$ and $\overline{\rho}_0$ of $v$, namely, by testing $\mathcal{J}_{\nu}'(u_n,v_n)$ with $(0,v_n\varphi_{\infty,\varepsilon})$ and arguing as before, we find
$\overline{\mu}_0-\lambda_2\overline{\eta}_0-\overline{\rho}_0\leq0$. This, combined with \eqref{HS_ineq_l}, produces 
\begin{equation}\label{ineq:cca2}
\overline{\mu}_0-\lambda_2\overline{\eta}_0\ge \mathcal{S}(\lambda_2,t,s)\overline{\rho}_0^{\frac{2}{2_{s}^*(t)}},
\end{equation}
from where we also conclude 
\begin{equation}\label{ineq:cc0a2}
\overline{\rho}_0=0\quad\text{or}\quad \overline{\rho}_0\ge \left[\mathcal{S}(\lambda_2,t,s)\right]^{\frac{N-t}{2s-t}}.
\end{equation}
In order to analyze the concentration phenomena at $\infty$, let us consider $\varphi_{\infty,\varepsilon}$ a cut-off function supported near $\infty$, that is
\begin{equation}\label{cutoffinfi}
\varphi_{\infty,\varepsilon}=0 \quad \hbox{in}\quad B_{R}(0),\quad \varphi_{\infty,\varepsilon}=1 \quad \hbox{in}
\quad B_{R+1}^c(0)\quad \hbox{and}\quad\displaystyle|\nabla \varphi_{\infty,\varepsilon}|\leq \frac{4}{\varepsilon},
\end{equation}
for $R>0$. We can similarly prove $\mu_{\infty}-\lambda_1\eta_{\infty}-\rho_{\infty}\leq 0$ and $\overline{\mu}_{\infty}-\lambda_2\overline{\eta}_{\infty}-\overline{\rho}_{\infty}\leq0$, by testing $\mathcal{J}_{\nu}'(u_n,v_n)$ with $(u_n\varphi_{\infty,\varepsilon},0)$ and $(0,v_n\varphi_{\infty,\varepsilon})$ respectively, so that
\begin{equation}\label{ineq:coninf_a}
\mu_{\infty}-\lambda_1\eta_{\infty}\ge \mathcal{S}(\lambda_1,t,s)\rho_{\infty}^{\frac{2}{2_{s}^*(t)}}\qquad\text{and}\qquad
\overline{\mu}_{\infty}-\lambda_2\overline{\eta}_{\infty}\ge \mathcal{S}(\lambda_2,s)\overline{\rho}_{\infty}^{\frac{2}{2_{s}^*}}.
\end{equation}
Thus, 
\begin{equation}\label{ineq:coninf}
\rho_{\infty}=0\quad\text{or}\quad \rho_{\infty}\ge \left[\mathcal{S}(\lambda_1,t,s)\right]^{\frac{N-t}{2s-t}}\qquad\text{and}\qquad
\overline{\rho}_{\infty}=0\quad\text{or}\quad \overline{\rho}_{\infty}\ge \left[\mathcal{S}(\lambda_2,t,s)\right]^{\frac{N-t}{2s-t}}.
\end{equation}
Next, given $\{(u_n,v_n)\}$ a PS sequence for $\mathcal{J}_{\nu}$ at level $c$, we have, up to a subsequence,
\begin{equation*}
\mathcal{J}_{\nu}(u_n,v_n)-\frac{1}{\alpha+\beta}\left\langle \mathcal{J}_{\nu}'(u_n,v_n)\left|\frac{(u_n,v_n)}{\|(u_n,v_n)\|_{\mathbb{D}}} \right.\right\rangle=c+\|(u_n,v_n)\|_{\mathbb{D}}\cdot o(1).
\end{equation*}
Hence,
\begin{equation}\label{eq:limit2}
c\!=\!\left(\! \frac{1}{2}\!-\! \frac{1}{\alpha+\beta}\! \right)\!\|(u_n,v_n)\|_{\mathbb{D}}^2\!+\!\left( \frac{1}{\alpha+\beta}\!-\!\frac{1}{2^*_{s}(t)}\! \right)\!\!\left(\!\int_{\mathbb{R}^N}\! \frac{|u_n|^{2^*_{s}(t)}}{|x|^{t}}dx  +  \int_{\mathbb{R}^N}\! \frac{|v_n|^{2^*_{s}(t)}}{|x|^{t}}dx\!\right)\!+o(1).
\end{equation}
Therefore, by using \eqref{con-comp}, \eqref{ineq:cca} and \eqref{ineq:coninf_a}, we get
\begin{equation}\label{ineq:l}
\begin{split}
c\ge& \left(\frac{1}{2}\!-\!\frac{1}{\alpha+\beta} \right)\!\!\Bigg(\!\|(\tilde{u},\tilde{v})\|_{\mathbb{D}}^2\! +\! (\mu_0-\lambda_1\eta_0)\!+\!(\mu_{\infty}-\lambda_1\eta_{\infty})\! +\!(\overline{\mu}_0-\lambda_2\overline{\eta}_0)\!+\!(\overline{\mu}_{\infty}-\lambda_2\overline{\eta}_{\infty})\!\Bigg)\\
&+\left(\frac{1}{\alpha+\beta}- \frac{1}{2_{s}^*(t)} \right)\left(\int_{\mathbb{R}^N}\frac{|\tilde{u}|^{2_{s}^*(t)}}{|x|^{t}}dx+  \rho_0+\rho_{\infty} +\int_{\mathbb{R}^N}\frac{|\tilde{v}|^{2_{s}^*(t)}}{|x|^{t}}dx +\overline{\rho}_0+\overline{\rho}_{\infty} \right)\\
\ge&\left(\frac{1}{2}-\frac{1}{\alpha+\beta} \right)\left( \mathcal{S}(\lambda_1,t,s)\left[\rho_0^{\frac{2}{2_{s}^*(t)}}+\rho_{\infty}^{\frac{2}{2_{s}^*(t)}}\right]+\mathcal{S}(\lambda_2,t,s)\left[\overline{\rho}_0^{\frac{2}{2_{s}^*(t)}}+\overline{\rho}_{\infty}^{\frac{2}{2_{s}^*(t)}}\right]\right)\\
&+\left(\frac{1}{\alpha+\beta}- \frac{1}{2_{s}^*(t)} \right)\left(\rho_0+\rho_{\infty} +\overline{\rho}_0+\overline{\rho}_{\infty} \right).
\end{split}
\end{equation}
In case of having concentration at $0$, i.e., if $\rho_0\neq0$, by \eqref{ineq:l} combined with \eqref{ineq:cc0a}, we reach a contradiction with the hypothesis on the energy level $c$, namely
\begin{equation*}
c\ge\frac{2s-t}{2(N-t)}\left[\mathcal{S}(\lambda_1,t,s)\right]^{\frac{N-t}{2s-t}}=\mathscr{C}(\lambda_1,t,s).
\end{equation*}
Then, $\rho_0=0$. Similarly, we get $\overline{\rho}_0=0$ and arguing as above using now \eqref{ineq:coninf} it also follows $\rho_{\infty}=0$ and $\overline{\rho}_{\infty}=0$. Thus, there exists a subsequence strongly converging in $L^{2_s^*(t)}(\mathbb{R}^N)\times L^{2_s^*(t)}(\mathbb{R}^N)$ such that
$\|(u_n-\tilde{u},v_n-\tilde{v})\|_{\mathbb{D}}^2 = \left\langle \mathcal{J}_\nu'(u_n,v_n)\big| (u_n-\tilde{u},v_n-\tilde{v}) \right\rangle + o(1)$ and, thus, the sequence $\{(u_n,v_n)\}$ strongly converges in $\mathbb{D}$ and the PS condition holds.
\end{proof}
To continue, we enhance Lemma \ref{lemmaPS2} by establishing the PS condition for energy levels that surpass the critical threshold \eqref{hyplemmaPS2}, while disregarding any multiples or combinations of the involved quantities related with the energy of the semi-trivial solutions. To that end, we exploit the concentration profile characterization provided in \cite[Theorem 2.1]{Bhakta2021}.
\begin{lemma}\label{lemmaPS1}
Assume that $\alpha+\beta<2_{s}^*(t)$, $\alpha\ge2$ and $\mathscr{C}(\lambda_1,t,s)\geq\mathscr{C}(\lambda_2,t,s)$. Then, there exists $\tilde{\nu}>0$ such that, if $0<\nu\leq\tilde{\nu}$ and $\{(u_n,v_n)\} \subset \mathbb{D}$ is a PS sequence for $\mathcal{J}^+_\nu$ at level $c\in\mathbb{R}$ such that
\begin{equation}\label{PS1}
\mathscr{C}(\lambda_1,t,s)<c<\mathscr{C}(\lambda_1,t,s)+\mathscr{C}(\lambda_2,t,s)
\end{equation}
and
\begin{equation}\label{PS2}
c\neq \ell \mathscr{C}(\lambda_2,t,s) \quad \mbox{ for every } \ell \in \mathbb{N}\setminus \{0\},
\end{equation}
then $(u_n,v_n)\to(\tilde{u},\tilde{v}) \in \mathbb{D}$ up to subsequence.
\end{lemma}
\begin{proof}
Arguing as in Lemma~\ref{lem_PSBounded}, any PS sequence for $\mathcal{J}_\nu^+$ is bounded in $\mathbb{D}$. Thus,  there exists a subsequence $\{(u_n,v_n)\}\rightharpoonup(\tilde{u},\tilde{v}) \in \mathbb{D}$. Moreover,  we can assume that $\{(u_n,v_n)\}$ is a non-negative PS sequence at level $c$ for $\mathcal{J}_\nu^+$. Actually, since $(\mathcal{J}^+_\nu)'(u_n,v_n)\to0$ as $n\to+\infty$, we get
\begin{equation*}
\begin{split}
\left\langle (\mathcal{J}^+_\nu)'(u_n,v_n){\big|} (u_n^-,0)\right\rangle&= \iint_{\mathbb{R}^{2N}}\frac{(u_n(x)-u_n(y))(u_n^-(x)-u_n^-(y))}{|x-y|^{N+2s}}dydx - \lambda_1 \int_{\mathbb{R}^N} \dfrac{(u_n^-)^2}{|x|^2} \, dx\\
&\geq \iint_{\mathbb{R}^{2N}}\frac{(u_n^-(x)-u_n^-(y))^2}{|x-y|^{N+2s}}dydx - \lambda_1 \int_{\mathbb{R}^N} \dfrac{(u_n^-)^2}{|x|^2} \, dx\\
&\geq \left(1-\frac{\lambda_1}{\Lambda_{N,s}}\right)\iint_{\mathbb{R}^{2N}}\frac{(u_n^-(x)-u_n^-(y))^2}{|x-y|^{N+2s}}dydx\to 0\qquad\ \text{as }n\to+\infty,
\end{split}
\end{equation*}
where we have used the fact $u^+u^-\equiv0$ together with the fractional Hardy inequality \eqref{hardy_inequality}. Thus, $u_n^-\to 0$ strongly in $\mathcal{D}^{s,2} (\mathbb{R}^N)$. Analogously, we can also deduce $v_n^-\to 0$.

As in the proof of Lemma~\ref{lemmaPS2}, the concentration \eqref{con-comp} holds up to a subsequence and for some positive numbers $\mu_0$, $\rho_0$, $\eta_0$, $\overline{\mu}_0$, $\overline{\rho}_0$ and $\overline{\eta}_0$ satisfying the inequalities \eqref{ineq:cca}, \eqref{ineq:cc0a}, \eqref{ineq:cca2} and \eqref{ineq:cc0a2}. The concentration at infinity is codified by the values $\mu_\infty$, $\rho_\infty$, $\overline{\mu}_\infty$ and $\overline{\rho}_\infty$ defined as in \eqref{con-compinfty} and such that \eqref{ineq:coninf_a} and \eqref{ineq:coninf} also hold.

Next, we prove that at least one of the components of the PS sequence strongly converges in $\mathcal{D}^{s,2}(\mathbb{R}^N)$. To that end, we start by proving that,
\begin{equation}\label{claimPS21}
\mbox{ either } u_n\to \tilde u\  \mbox{  strongly in } L^{2_s^*(t)}(\mathbb{R}^N) \qquad \mbox{or} \qquad v_n\to \tilde v\ \mbox{  strongly in }L^{2_s^*(t)}(\mathbb{R}^N).
\end{equation}
Assume, by contradiction, that $\{u_n\}$ and $\{v_n\}$ do not strongly converge in $L^{2_s^*(t)}(\mathbb{R}^N)$ so that $\rho_{j}>0$ and $\overline{\rho}_k>0$ for $j,k\in \{0, \infty\}$. However, this assumption produces a contradiction with \eqref{PS1} since, as in \eqref{eq:limit2}, by using \eqref{con-comp}, \eqref{ineq:cca}, \eqref{ineq:cc0a}, \eqref{ineq:cca2},  \eqref{ineq:cc0a2}, \eqref{ineq:coninf_a} and \eqref{ineq:coninf} applied in \eqref{ineq:l} we get
\begin{equation*}
\begin{split}
c=&\left( \frac{1}{2} -\frac{1}{\alpha+\beta} \right)\|(u_n,v_n)\|_{\mathbb{D}}^2+\left( \frac{1}{\alpha+\beta}-\frac{1}{2^*_{s}(t)} \right)\left(\int_{\mathbb{R}^N} \frac{|u_n|^{2^*_{s}(t)}}{|x|^{t}}dx + \int_{\mathbb{R}^N} \frac{|v_n|^{2^*_{s}(t)}}{|x|^{t}}dx\right)+o(1)\\
\ge&\left( \frac{1}{2} -\frac{1}{\alpha+\beta} \right)\left(\mathcal{S}(\lambda_1,t,s)\rho_j^{\frac{2}{2_{s}^*(t)}} +\mathcal{S}(\lambda_2,t,s) \overline{\rho}_k^{\frac{2}{2_{s}^*(t)}}\right)+\left( \frac{1}{\alpha+\beta}-\frac{1}{2^*_{s}(t)} \right)(\rho_j+\overline{\rho}_k) \\
\ge&\, \frac{2s-t}{2(N-t)}\left(\left[\mathcal{S}(\lambda_1,t,s)\right]^{\frac{N-t}{2s-t}}+\left[\mathcal{S}(\lambda_2,t,s)\right]^{\frac{N-t}{2s-t}}\right)\\
=&\,\mathscr{C}(\lambda_1,t,s)+\mathscr{C}(\lambda_2,t,s).
\end{split}
\end{equation*}
Thus \eqref{claimPS21} is proved. Assume that the sequence $\{u_n\}$ strongly converges in $L^{2_s^*(t)}(\mathbb{R}^N)$. Since
$
\|u_n-\tilde u\|_{\lambda_1}^2 = \left\langle \mathcal{J}_\nu'(u_n,v_n){\big|} (u_n-\tilde u,0) \right\rangle + o(1),
$
it follows that $u_n\to \tilde u$ in $\mathcal{D}^{s,2}(\mathbb{R}^N)$. Repeating the argument for $\{v_n\}$ we conclude 
\begin{equation*}
\mbox{ either } u_n\to  \tilde u\  \mbox{ in } \mathcal{D}^{s,2}(\mathbb{R}^N) \qquad \mbox{ or } \qquad v_n\to  \tilde v \ \mbox{ in } \mathcal{D}^{s,2}(\mathbb{R}^N).
\end{equation*}

To complete the proof, we prove that actually both $\{u_n\}$ and $\{v_n\}$ strongly converge in $\mathcal{D}^{s,2}(\mathbb{R}^N)$. The proof is splitted into two cases.\newline

\textbf{Case 1}: The sequence $\{v_n\}$ strongly converges to $\tilde{v}$ in $\mathcal{D}^{s,2}(\mathbb{R}^N)$.\hfill\newline 
Our goal is then to prove that $\{u_n\}$ strongly converges to $\tilde{u}$ in $\mathcal{D}^{s,2}(\mathbb{R}^N)$. 
First, we prove that the sequence $\{v_n\}$ strongly converges to a nontrivial limit $\tilde{v}\not \equiv 0$. Assume by contradiction that $\tilde{v}\equiv 0$, then $\tilde{u}\ge 0 $ and $\tilde{u}$ verifies \eqref{eq:uncoupled} with $j=1$. Thus, $\tilde{u}=z_\mu^{(1)}$ for some $\mu>0$ and $\tilde{u}$ satisfies \eqref{relation}, i.e. 
\[\int_{\mathbb{R}^N}\frac{ \tilde{u}^{2_{s}^*(t)}}{|x|^{t}} dx =[ \mathcal{S}(\lambda_1,t,s)]^{\frac{N-t}{2s-t}}.\]
Next, assume that none of the subsequences of $\{u_n\}$ converge. A concentration at more than one point, i.e., at $0$ and $\infty$, produces a contradiction with \eqref{PS1} since, because of \eqref{ineq:cca}, \eqref{ineq:cc0a}, \eqref{ineq:coninf_a}, \eqref{ineq:coninf} and \eqref{ineq:l}, it follows that
\begin{equation*}
c\ge \frac{2s-t}{N-t}\left[\mathcal{S}(\lambda_1,t,s)\right]^{\frac{N-t}{2s-t}}=2\mathscr{C}(\lambda_1,t,s)\ge \mathscr{C}(\lambda_1,t,s)+\mathscr{C}(\lambda_2,t,s).
\end{equation*}
Thus, the sequence $\{u_n\}$ concentrates only at $0$ or at $\infty$. Nevertheless, if $\{u_n\}$ concentrates at one point we get a contradiction with \eqref{PS1} since, by \eqref{ineq:l} jointly with \eqref{ineq:cca}, \eqref{ineq:cc0a} and the hypothesis $\mathscr{C}(\lambda_1,t,s)\geq\mathscr{C}(\lambda_2,t,s)$ we have
\begin{equation}\label{eq:0}
\begin{split}
c&\ge \frac{2s-t}{2(N-t)}\left( \int_{\mathbb{R}^N}\frac{ \tilde{u}^{2_{s}^*(t)}}{|x|^{t}} dx +\left[\mathcal{S}(\lambda_1,t,s)\right]^{\frac{N-t}{2s-t}}  \right)=2\mathscr{C}(\lambda_1,t,s) \ge \mathscr{C}(\lambda_1,t,s)+\mathscr{C}(\lambda_2,t,s).
\end{split}
\end{equation}
In case of having $\tilde{v}\equiv 0$ and $\tilde{u}\equiv 0$, then
\begin{equation*}
(-\Delta)^s u_n - \lambda_1 \frac{u_n}{|x|^{2s}}-\frac{u_n^{2_{s}^*(t)-1}}{|x|^{t}}=o(1) \qquad \mbox{ in the dual space } \left( \mathcal{D}^{s,2} (\mathbb{R}^N)\right)',
\end{equation*}
and, as $\{u_n\}$ concentrates at most one point, namely at $0$ or at $\infty$, it follows that 
\begin{equation}\label{eq:1}
c=\mathcal{J}_\nu(u_n,v_n)+o(1)=\frac{2s-t}{2(N-t)} \int_{\mathbb{R}^N} \frac{ u_n^{2_{s}^*(t)}}{|x|^{t}}dx +o(1)\to \frac{2s-t}{2(N-t)} \rho_j.
\end{equation}
As the concentration takes place at $0$ or $\infty$, the sequence $\{u_n\}$ is a positive PS sequence for $\mathcal{J}_1(u)$, defined in \eqref{funct:Ji}. Because of the characterization of PS sequences for the functional $J_j$ provided in \cite[Theorem 2.1]{Bhakta2021}, it follows that $\rho_j\ge l [\mathcal{S}(\lambda_1,t,s)]^{\frac{N-t}{2s-t}}$ and, thus, from \eqref{eq:1} we conclude
\begin{equation*}
\begin{split}
c&=\mathcal{J}_\nu(u_n,v_n)+o(1)=\mathcal{J}_1(u_n)+o(1)\to\ell \mathscr{C}(\lambda_1,t,s),
\end{split}
\end{equation*}
with $\ell \in \mathbb{N}\cup\{0\}$ contradicting \eqref{PS1}. Then, we conclude the strong convergence of $\{v_n\}$ to a nontrivial limit $\tilde{v}\gneq 0$ in $\mathbb{R}^N$.\newline
To continue, we prove that $u_n\rightharpoonup\tilde{u}$ in $\mathcal{D}^{s,2}(\mathbb{R}^N)$ such that $\tilde{u}\not \equiv 0$. If one assumes $\tilde{u}=0$, arguing as above we also get $\tilde{v}=z_{\mu}^{(2)}$ so that, as in \eqref{eq:0}, we get a contradiction with \eqref{PS1}. Thus, $\tilde{u},\tilde{v} \gneq 0$. Next, taking $n\to+\infty$ in the equality
\begin{equation*}
\begin{split}
c=&\ \mathcal{J}_\nu(u_n,v_n)-\frac{1}{2}\left\langle \mathcal{J}_\nu'(u_n,v_n){\big|} (u_n,v_n) \right\rangle + o(1)\\
 =&\  \frac{2s-t}{2(N-t)}\left( \int_{\mathbb{R}^N} \frac{ u_n^{2_{s}^*(t)}}{|x|^{t}}dx+\int_{\mathbb{R}^N} \frac{ v_n^{2_{s}^*(t)}}{|x|^{t}}dx\right) +\nu \left( \frac{\alpha+\beta-2}{2} \right)  \int_{\mathbb{R}^N} h(x)\frac{u_n^{\alpha}\, v_n^{\beta}}{|x|^t} dx + o(1),
\end{split}
\end{equation*}
we find, for $j \in \{0, \infty\}$,
\begin{equation}\label{eqlemmaPS10}
\begin{split}
c=&  \  \frac{2s-t}{2(N-t)}\left( \int_{\mathbb{R}^N} \frac{ \tilde{u}^{2_{s}^*(t)}}{|x|^{t}}dx+ \rho_j+  \int_{\mathbb{R}^N} \frac{ \tilde{v}^{2_{s}^*(t)}}{|x|^{t}}dx\right)+\nu \left( \frac{\alpha+\beta-2}{2} \right)  \int_{\mathbb{R}^N} h(x)  \frac{\tilde{u}^{\alpha}\, \tilde{v}^{\beta}}{|x|^t} dx.
\end{split}
\end{equation}
On the other hand, as $\left\langle \mathcal{J}_\nu'(u_n,v_n){\big|} (\tilde{u},\tilde{v}) \right\rangle \to 0$ as $n\to\infty$, we also get
\begin{equation*}
\|(\tilde{u},\tilde{v}) \|_{\mathbb{D}}=  \int_{\mathbb{R}^N}\frac{ \tilde{u}^{2_{s}^*(t)}}{|x|^{t}}dx+\int_{\mathbb{R}^N}\frac{ \tilde{v}^{2_{s}^*(t)}}{|x|^{t}}dx  + \nu(\alpha+\beta) \int_{\mathbb{R}^N} h(x) \frac{\tilde{u}^\alpha \tilde{v}^\beta}{|x|^t} dx,
\end{equation*}
which is equivalent to say that $(\tilde{u},\tilde{v}) \in\mathcal{N}_\nu$. Then,
\begin{equation}\label{eqlemmaPS11}
\begin{split}
 \mathcal{J}_\nu(\tilde{u},\tilde{v})=&\ \frac{2s-t}{2(N-t)}\left( \int_{\mathbb{R}^N} \frac{ \tilde{u}^{2_{s}^*(t)}}{|x|^{t}}dx+ \int_{\mathbb{R}^N} \frac{ \tilde{v}^{2_{s}^*(t)}}{|x|^{t}}dx\right) + \nu \left( \frac{\alpha+\beta-2}{2} \right)  \int_{\mathbb{R}^N} h\frac{\tilde{u}^{\alpha}\, \tilde{v}^{\beta}}{|x|^t} dx\\
 \le&\ c<\mathscr{C}(\lambda_1,t,s)+\mathscr{C}(\lambda_2,t,s).
 \end{split}
\end{equation}
Using \eqref{eqlemmaPS10}, \eqref{eqlemmaPS11}, \eqref{Nnueq} \eqref{ineq:cca}, \eqref{ineq:cc0a}, \eqref{ineq:l} and \eqref{PS1} we get that
\begin{equation*}
\mathcal{J}_\nu(\tilde{u},\tilde{v})= c - \frac{2s-t}{(N-t)}\rho_j
<\ \mathscr{C}(\lambda_1,t,s)+\mathscr{C}(\lambda_2,t,s) - \frac{2s-t}{2(N-t)}\left[\mathcal{S}(\lambda_1,t,s)\right]^{\frac{N-t}{2s-t}}=\mathscr{C}(\lambda_2,t,s).
\end{equation*}
The above expression implies that
$$
\tilde{c}_\nu= \inf_{(u,v)\in\mathcal{N}_\nu} \mathcal{J}_\nu(u,v) < \mathscr{C}(\lambda_2,t,s).
$$
However, for $\nu$ sufficiently small, Theorem~\ref{thm:groundstatesalphabeta} states that $\tilde c_\nu =  \mathscr{C}(\lambda_2,t,s)$, which contradicts the former inequality. Thus, we have proved that $u_n\to \tilde{u}$ strongly in $\mathcal{D}^{s,2}(\mathbb{R}^N)$.

\textbf{Case 2}: The sequence $\{u_n\}$ strongly converges to $\tilde{u}$ in $\mathcal{D}^{s,2}(\mathbb{R}^N)$.\hfill\newline
Now, we prove that $\{v_n\}$ strongly converges to $\tilde{v}$ in $\mathcal{D}^{s,2}(\mathbb{R}^N)$. As before, we start by proving $\tilde{u}\not \equiv 0$. Assume $\tilde{u} \equiv 0$ so that $\{v_n\}$ is a PS sequence for the energy functional $\mathcal{J}_2$ defined in \eqref{funct:Ji} at energy level $c$. We continue by proving that a concentration for $\{v_n\}$ under the assumption $\tilde{u} \equiv 0$ produces a contradiction with the energy assumption \eqref{PS2}. Suppose that none of the subsequences of $\{v_n\}$ converge. Since $v_n\rightharpoonup\tilde{v}$ in $\mathcal{D}^{s,2}(\mathbb{R}^N)$, then $\tilde{v}$ satisfies the \eqref{eq:uncoupled}. By the characterization of the PS sequences provided in \cite[Theorem 2.1]{Bhakta2021}, one has
\begin{equation*}
\begin{split}
c& = \lim_{n\to+\infty} \mathcal{J}_2 (v_n) =  \mathcal{J}_2 (z_\mu^{(2)})+\ell\mathscr{C}(\lambda_2,t,s),
\end{split}
\end{equation*}
for some $\ell \in \mathbb{N}\cup\{0\}$. This clearly contradicts \eqref{PS2} as, in case of having $\tilde{v}\equiv0$, then $c=\ell \mathscr{C}(\lambda_2,t,s)$ and, if $\tilde{v}\not\equiv0$, then $\tilde{v}=z_\mu^{(2)}$ for some $\mu>0$ so that, by \eqref{csemi} and \eqref{csemi2}, we find $c=(\ell+1) \mathscr{C}(\lambda_2,t,s)$. Hence, $\tilde{u} \not \equiv 0$. Analogously, a concentration argument rules out the case $\tilde v \equiv 0$ by producing a contradiction with \eqref{PS1} since, under this assumption, $\tilde{u}$ solves \eqref{eq:uncoupled} so that $u=z_\mu^{(1)}$ for some $\mu>0$ and, thus,
\begin{equation*}
\begin{split}
c&\ge \frac{2s-t}{2(N-t)}\left( \int_{\mathbb{R}^N}\frac{ \tilde{u}^{2_{s}^*(t)}}{|x|^{t}} dx +\left[\mathcal{S}(\lambda_2,t,s)\right]^{\frac{N-t}{2s-t}}  \right)= \mathscr{C}(\lambda_1,t,s)+\mathscr{C}(\lambda_2,t,s).
\end{split}
\end{equation*}
Therefore the limit is a nontrivial pair $\tilde{u},\tilde{v}\not \equiv 0$. To finish the proof we prove that no concentration phenomena can take place for $\{v_n\}$. Suppose that $v_n$ does not strongly converge in $\mathcal{D}^{s,2}(\mathbb{R}^N)$, i.e., there exists at least one $k\in \{0, \infty\}$ such that $\overline{\rho}_k>0$. Since, by \eqref{eqlemmaPS10}, we have
\begin{equation*}
\begin{split}
c=&\ \frac{2s-t}{2(N-t)}\left( \int_{\mathbb{R}^N}\frac{ \tilde{u}^{2_{s}^*(t)}}{|x|^{t}} dx +\int_{\mathbb{R}^N}\frac{|\tilde{v}|^{2_{s}^*(t)}}{|x|^{t}}dx +\overline{\rho}_0+\overline{\rho}_{\infty} \right) +\nu \left( \frac{\alpha+\beta-2}{2} \right)  \int_{\mathbb{R}^N} h(x)  \frac{\tilde{u}^{\alpha}\, \tilde{v}^{\beta}}{|x|^t} dx,
\end{split}
\end{equation*}
because of \eqref{eqlemmaPS11}, \eqref{ineq:cca}, \eqref{ineq:cc0a} and \eqref{PS1}, it follows that
\begin{equation}\label{eqlemmaPS12}
\begin{split}
\mathcal{J}_\nu(\tilde{u},\tilde{v})&= c - \frac{2s-t}{2(N-t)}\left( \overline{\rho}_0 + \overline{\rho}_\infty\right) < \mathscr{C}(\lambda_1,t,s)+\mathscr{C}(\lambda_2,t,s) - \frac{2s-t}{2(N-t)}\left[\mathcal{S}(\lambda_2,t,s)\right]^{\frac{N-t}{2s-t}}\\
& =  \mathscr{C}(\lambda_1,t,s). 
\end{split}
\end{equation}
Because of \eqref{HS_ineq_l} and the first equation of \eqref{system},
\begin{equation}\label{eqlemmaPS13}
\begin{split}
\int_{\mathbb{R}^N} \frac{\tilde{u}^{2_{s}^*(t)}}{|x|^{t}} \, dx + \nu \int_{\mathbb{R}^N} h(x)\frac{\tilde{u}^{\alpha}\, \tilde{v}^{\beta}}{|x|^t} dx 
 &\ge \mathcal{S}(\lambda_1,t,s) \left(\int_{\mathbb{R}^N} \frac{\tilde{u}^{2_{s}^*(t)}}{|x|^{t}} \, dx\right)^{\frac{2}{2_{s}^*(t)}}.
\end{split}
\end{equation}
Since we are in the subcritical regime, $\alpha+\beta<2_s^*$, by\eqref{H} we have $h\in L^1(\mathbb{R}^N)\cap L^{\infty}(\mathbb{R}^N)$ so that, by H\"older's inequality, we have
\begin{equation}\label{Holder}
 \int_{\mathbb{R}^N} h(x)\frac{\tilde{u}^{\alpha}\tilde{v}^{\beta}}{|x|^t} dx
  \leq C(h)  \left( \int_{\mathbb{R}^N} \frac{\tilde{u}^{2^*_{s}(t)}}{|x|^{t}} dx\right)^{\frac{\alpha}{2^*_{s}(t)}}  \left( \int_{\mathbb{R}^N} \frac{\tilde{v}^{\beta }}{|x|^{t}} dx\right)^{\frac{\beta}{2^*_{s}(t)}}. 
\end{equation}
By combining \eqref{eqlemmaPS11} and \eqref{Holder}, from \eqref{eqlemmaPS13} it follows that
\begin{equation}\label{eqlemmaPS14}
\sigma_1 + C \nu \left(\int_{\mathbb{R}^N} \frac{\tilde{u}^{2_{s}^*(t)}}{|x|^{t}} \, dx\right)^{\frac{\alpha}{2_{s}^*(t)}}\ge \mathcal{S}(\lambda_1,t,s) \left(\int_{\mathbb{R}^N} \frac{\tilde{u}^{2_{s}^*(t)}}{|x|^{t}} \, dx\right)^\frac{2}{2_{s}^*(t)}.
\end{equation}
Moreover, since $\tilde{v}\not \equiv 0$, we can choose $\tilde \varepsilon>0$, such that
\begin{equation*}
\frac{2s-t}{2(N-t)}\int_{\mathbb{R}^N}\frac{ \tilde{v}^{2_{s}^*(t)}}{|x|^{t}} \, dx\ge \tilde \varepsilon.
\end{equation*} 
Then, taking $\varepsilon>0$ such that $\tilde \varepsilon\ge \varepsilon \mathscr{C}(\lambda_1,t,s)$, by \eqref{eqlemmaPS14} and Lemma~\ref{lem_algebr} with
$\sigma\vcentcolon= \int_{\mathbb{R}^N} \frac{\tilde{u}^{2_{s}^*(t)}}{|x|^{t}} \, dx,$
there exists $\tilde \nu>0$ such that, for any $0<\nu\leq \tilde{\nu}$,
\begin{equation*}
\int_{\mathbb{R}^N} \frac{\tilde{u}^{2_{s}^*(t)}}{|x|^{t}} \, dx\ge (1-\varepsilon)[\mathcal{S}(\lambda_1,t,s)]^{\frac{N-t}{2s-t}}.
\end{equation*}
The former estimates combined with \eqref{eqlemmaPS11} produces
\begin{equation*}
\mathcal{J}_\nu(\tilde{u},\tilde{v}) \ge (1-\varepsilon)\frac{2s-t}{2(N-t)}[\mathcal{S}(\lambda_1,t,s)]^{\frac{N-t}{2s-t}} + \tilde{\varepsilon} =\mathscr{C}(\lambda_1,t,s),
\end{equation*}
in contradiction with \eqref{eqlemmaPS12}. Thus, $v_n\to \tilde{v}$ strongly in $\mathcal{D}^{s,2}(\mathbb{R}^N)$.
\end{proof}
The case $\beta\ge2$, and $\mathscr{C}(\lambda_1,t,s)\geq\mathscr{C}(\lambda_2,t,s)$ can be derived in a similar way.
\begin{lemma}
Assume that $\alpha+\beta<2_{s}^*(t)$, $\beta\ge2$, and $\mathscr{C}(\lambda_1,t,s)\geq\mathscr{C}(\lambda_2,t,s)$. Then, there exists $\tilde{\nu}>0$ such that, if $0<\nu\leq\tilde{\nu}$ and $\{(u_n,v_n)\} \subset \mathbb{D}$ is a PS sequence for $\mathcal{J}^+_\nu$ at level $c\in\mathbb{R}$ such that
\begin{equation}\label{PS2a}
\mathscr{C}(\lambda_2,t,s)<c<\mathscr{C}(\lambda_1,t,s)+\mathscr{C}(\lambda_2,t,s),
\end{equation}
and
\begin{equation}\label{PS2aa}
c\neq \ell \mathscr{C}(\lambda_1,t,s) \quad \mbox{ for every } \ell \in \mathbb{N}\setminus \{0\},
\end{equation}
then $(u_n,v_n)\to(\tilde{u},\tilde{v}) \in \mathbb{D}$ up to subsequence.
\end{lemma}

\subsection{Critical range $\alpha+\beta=2_s^*(t)$}\hfill\newline
In the critical regime the coupling term comes into play by contributing to the energy of the functional $J_\nu$ through a concentration phenomena. We exclude this posibility by extending Lemmas \ref{lemmaPS2} and \ref{lemmaPS1} to deal with $\alpha+\beta=2_s^*(t)$. In order to cancel out the contribution produced by a concentration at the origin of the coupling term we further assume 
\begin{equation}\label{H_1}\tag{$H_1$}
h \mbox{ continuous around $0$ and $\infty$ and }h(0)=\lim_{x\to+\infty} h(x)=0.
\end{equation}
Let us also stress that, due the presence of Hardy-Sobolev critical couplings in system \ref{system}, the concentration take places only at $0$ or at $\infty$. In case of considering $t=0$, the concentration of the coupling term can take place at points $x_j\in\mathbb{R}^N$ different from $0$ or $\infty$. In order to exclude this possibility one can work in the radial setting (so that the concentration is forced to be at $0$ and then hypothesis \eqref{H_1} completes the work) or one can further assume the extra hypothesis $\nu$ small enough, (cf. \cite{Abdellaoui2009}, \cite{Colorado2022}).
\begin{lemma}\label{lemcritic} 
Assume that  $\alpha+\beta=2_{s}^*(t)$ and hypothesis \eqref{H_1} hold. Let $\{(u_n,v_n)\} \subset \mathbb{D}$ be a PS sequence for $\mathcal{J}_\nu$ at level $c\in\mathbb{R}$ such that
\begin{itemize}
\item[i)] either $c$ satisfies \eqref{hyplemmaPS2},
\item[ii)] or $c$ satisfies \eqref{PS1} and \eqref{PS2} if $\alpha\ge 2$ and  $\mathscr{C}(\lambda_1,t,s)\geq\mathscr{C}(\lambda_2,t,s)$,
\item[iii)] or $c$ satisfies \eqref{PS2a} and \eqref{PS2aa} if $\beta\ge 2$ and  $\mathscr{C}(\lambda_1,t,s)\leq\mathscr{C}(\lambda_2,t,s)$.
\end{itemize}
Then, there exists $\tilde{\nu}>0$ such that, for every $0<\nu\leq \tilde{\nu}$, the sequence $(u_n,v_n)\to(\tilde{u},\tilde{v}) \in \mathbb{D}$ up to a subsequence.
\end{lemma}
\begin{proof}
Following the proof of Lemmas \ref{lemmaPS2} and \ref{lemmaPS1}, in order to avoid concentration at the origin, it is sufficient to show (see \eqref{eq:cc1}) that
\begin{equation}\label{radial:at0}
\lim\limits_{\varepsilon\to0}\limsup\limits_{n\to+\infty}\int_{\mathbb{R}^N}h(x)\frac{|u_n|^\alpha |v_n|^{\beta}}{|x|^{t}}\varphi_{0,\varepsilon}(x)dx=0,
\end{equation}
with $\varphi_{0,\varepsilon}$ as in \eqref{cutoff}. Similarly, the concentration at $\infty$ is excluded once we show,
\begin{equation}\label{radial:atinf}
\lim\limits_{R\to+\infty}\limsup\limits_{n\to+\infty}\int_{|x|>R}h(x)\frac{|u_n|^\alpha |v_n|^{\beta}}{|x|^{t}}\varphi_{\infty,\varepsilon}(x)dx=0,
\end{equation}
with $\varphi_{\infty,\varepsilon}$ as in \eqref{cutoffinfi}. Applying H\"older's inequality and using $\alpha+\beta=2_s^*(t)$, we find
\begin{equation*}
\begin{split}
\int_{\mathbb{R}^N} h\frac{|u_n|^{\alpha}|v_n|^{\beta}}{|x|^t} \varphi_{0,\varepsilon}dx
& \leq\left(\int_{\mathbb{R}^N}h\frac{|u_n|^{2_{s}^*(t)}}{|x|^{t}}\varphi_{0,\varepsilon}dx\right)^{\frac{\alpha}{2_{s}^*(t)}}\left(\int_{\mathbb{R}^N}h\frac{|v_n|^{2_{s}^*(t)}}{|x|^{t}}\varphi_{0,\varepsilon}dx\right)^{\frac{\beta}{2_{s}^*(t)}}.
\end{split}
\end{equation*}
 But, by \eqref{con-comp} and \eqref{H_1}, we have
\begin{equation*}
\lim\limits_{n\to+\infty}\int_{\mathbb{R}^N}h(x)\frac{|u_n|^{2_{s}^*(t)}}{|x|^{t}}\varphi_{0,\varepsilon}dx=\int_{\mathbb{R}^N}h(x)\frac{|\tilde{u}|^{2_{s}^*(t)}}{|x|^{t}}\varphi_{0,\varepsilon}\,dx+\rho_0h(0)
\leq\int_{|x|\leq\varepsilon}h(x)\frac{|\tilde{u}|^{2_{s}^*(t)}}{|x|^{t}} \, dx,
\end{equation*}
and 
\begin{equation*}
\lim\limits_{n\to+\infty}\int_{\mathbb{R}^N}h(x)\frac{|v_n|^{2_{s}^*(t)}}{|x|^{t}}\varphi_{0,\varepsilon}dx=\int_{\mathbb{R}^N}h(x)\frac{|\tilde{v}|^{2_{s}^*(t)}}{|x|^{t}}\varphi_{0,\varepsilon}\,dx+\overline{\rho}_0h(0)\leq\int_{|x|\leq\varepsilon}h(x)\frac{|\tilde{v}|^{2_{s}^*(t)}}{|x|^{t}} \, dx,
\end{equation*}
so \eqref{radial:at0} follows. The proof of \eqref{radial:atinf} follows similarly by using now $\lim\limits_{|x|\to+\infty}h(x)=0$.
\end{proof}

\section{Proofs of main Results}\label{sect:Main-Results}
Once we have ensured the PS condition under a {\it quantization} of the energy levels, we can prove the main results concerning the existence of bound and ground states to \eqref{system}.
\begin{theorem}\label{thm:nugrande}
Assume either $\alpha+\beta<2^*_s(t)$ or $\alpha+\beta=2^*_s(t)$ satisfying \eqref{H_1}. Then there exists $\overline{\nu}>0$ such that the system \eqref{system} has a positive ground state $(\tilde{u},\tilde{v}) \in \mathbb{D}$ for every $\nu>\overline{\nu}$.
\end{theorem}
\begin{proof}[Proof of Theorem~\ref{thm:nugrande}]
First, note that, given $(u,v)\in\mathbb{D}\setminus \{(0,0)\}$, we can always take $t=t_{(u,v)}$ such that $(tu,tv) \in \mathcal{N}_\nu$. Actually, $t=t_{(u,v)}$ is the solution to the algebraic equation
\begin{equation}\label{normH}
\|(u,v)\|_\mathbb{D}^2=\ \tau^{2_{s}^*(t)-2}\left(\int_{\mathbb{R}^N} \frac{|u|^{2^*_{s}(t)}}{|x|^{t}}dx + \int_{\mathbb{R}^N} \frac{|v|^{2^*_{s}(t)}}{|x|^{t}} dx\right)+ \nu (\alpha+\beta) \, \tau^{\alpha+\beta-2} \int_{\mathbb{R}^N} h \frac{ |u|^{\alpha} |v|^{\beta}}{|x|^t}dx.
\end{equation}
Then, fixed $(u,v)\in\mathbb{D}\setminus \{(0,0)\}$,  take $t$ such that $(tu,tv) \in \mathcal{N}_\nu$, i.e., $t$ satisfying \eqref{normH} above. Since $\alpha+\beta>2$, from \eqref{normH}, we have
\begin{equation*}
\lim_{\nu \to+\infty} t_\nu^{\alpha+\beta-2} \nu = \dfrac{\|(u,v)\|_\mathbb{D}^2}{ \displaystyle \int_{\mathbb{R}^N} h(x) \dfrac{|u|^{\alpha} |v|^{\beta}}{|x|^t}  \, dx},
\end{equation*}
so that $t=t_\nu \to 0$ as $\nu\to+\infty$ and, thus, 
\begin{equation*}
\mathcal{J}_\nu(t_\nu u ,t_\nu v)=\left(\frac{1}{2}-\frac{1}{\alpha+\beta} +o(1)  \right) t_\nu^2 \|(u,v)\|^2_\mathbb{D}.
\end{equation*}
Hence $\mathcal{J}_\nu(t_\nu u ,t_\nu v)=o(1)$ as $\nu\to+\infty$ and there exists $\overline{\nu}>0$ large enough such that, if $\nu>\overline{\nu}$, then
\begin{equation}\label{minimumlevel}
\tilde{c}_\nu=\inf_{(u,v) \in \mathcal{N}_\nu} \mathcal{J}_\nu (u,v)< \min \{\mathcal{J}_\nu(z_\mu^{(1)},0),\mathcal{J}_\nu(0,z_\mu^{(2)}) \}=\min \{ \mathscr{C}(\lambda_1,t,s), \mathscr{C}(\lambda_2,t,s)\}.
\end{equation}
If $\alpha+\beta < 2^*_{s}(t)$, from Lemma~\ref{lemmaPS2}, we get the existence of $(\tilde{u},\tilde{v}) \in \mathbb{D}$ such that $\mathcal{J}_\nu(\tilde{u},\tilde{v})=\tilde{c}_\nu$. We finish by proving that both $\tilde{u}$ and $\tilde{v}$ are indeed positive. Let $(|\tilde{u}|,|\tilde{v}|)\in\mathbb{D}$ and take $\tilde{\tau}\in\mathbb{R}$ such that $(\tilde{\tau}|\tilde{u}|,\tilde{\tau}|\tilde{v}|)\in\mathcal{N}_{\nu}$. Actually, $\tilde{\tau}$ solves \eqref{normH} with $(|\tilde{u}|,|\tilde{v}|)$ from where we get $\tilde{\tau}>0$. On the other hand, since $(\tilde{u},\tilde{v})\in\mathcal{N}_{\nu}$, that is, the pair $(\tilde{u},\tilde{v})\in\mathcal{N}_{\nu}$ satisfies \eqref{Nnueq1}, together with the inequality $\|(|\tilde{u}|,|\tilde{v}|)\|_{\mathbb{D}}\leq\|(\tilde{u},\tilde{v})\|_{\mathbb{D}}$ we also get  $\tilde{\tau}\leq 1$. Next, let us set $g(\tau)=J_\nu(\tau u,\tau v)$. Because of \eqref{alphabeta}, we have $g'''(\tau)<0$ and $g'(0)=0$ and $g(\tau)\to-\infty$ as $\tau\to+\infty$. Then, there exists $\tau_0>\tau_1>0$ such that $g'(\tau)$ attains its maximum at $\tau_0$ and $g'(\tau_1)=0$. Moreover, $g''(\tau)<0$ for $\tau>\tau_0$, in particular $g''(\tau_1)<0$. Thus, since $(\tilde{u},\tilde{v})\in\mathcal{N}_{\nu}$ is the unique maximum point of $g(t)=J_\nu(\tau \tilde{u},\tau \tilde{v})$, we find
\[\tilde{c}_{\nu}=J_{\nu}(\tilde{u},\tilde{v})|\geq\max\limits_{\tau>\tau_0}J_{\nu}(\tau\tilde{u},\tau\tilde{v})\geq J_{\nu}(\tilde{\tau}\tilde{u},\tilde{\tau}\tilde{v})\geq J_{\nu}(\tilde{\tau}|\tilde{u}|,\tilde{\tau}|\tilde{v}|)\geq \tilde{c}_{\nu},\]
so we can assume $\tilde{u}\ge 0$ and $\tilde{v}\ge 0$ in $\mathbb{R}^N$. Moreover,
$\tilde{u}\not\equiv0$ and $\tilde{v}\not\equiv$ since, otherwise, if we suppose that $\tilde{u}\equiv0$ it follows that $\tilde{v}$ is a solution to \eqref{eq:uncoupled} and further $\tilde{v}=z_\mu^{(2)}$ which contradicts \eqref{minimumlevel}. The same argument works for $\tilde{v}\equiv0$. Because of the maximum principle given in \cite[Theorem 1.2]{DelPezzo2017} we conclude $\tilde{u},\tilde{v}>0$ in $\mathbb{R}^N\backslash\{0\}$.
\end{proof}
We continue by proving the existence of positive ground state solutions based on the order relation between $\mathscr{C}(\lambda_1,t,s)$ and $\mathscr{C}(\lambda_2,t,s)$ combined with the cases $\alpha\leq2$ or $\beta\leq2$.
\begin{theorem}\label{thm:lambdaground}
Assume either $\alpha+\beta<2^*_s$ or $\alpha+\beta=2^*_s$ satisfying \eqref{H_1}. If one of the following statements is satisfied
\begin{itemize}
\item[$i)$] $\mathscr{C}(\lambda_2,t,s)\geq\mathscr{C}(\lambda_1,t,s)$ and either $\beta=2$ and $\nu$ large enough or $\beta<2$,
\item[$ii)$] $\mathscr{C}(\lambda_1,t,s)\geq\mathscr{C}(\lambda_2,t,s)$ and either $\alpha=2$ and $\nu$ large enough or $\alpha<2$,
\end{itemize}
then system \eqref{system} admits a positive ground state $(\tilde{u},\tilde{v})\in\mathbb{D}$.

In particular, if $\max\{\alpha,\beta\}<2$ or $\max\{\alpha,\beta\}\le2$ with $\nu$ sufficiently large, then system \eqref{system} admits a positive ground state $(\tilde{u},\tilde{v})\in\mathbb{D}$.
\end{theorem}

\begin{proof}[Proof of Theorem~\ref{thm:lambdaground}]
Assuming hypothesis $i)$, from Proposition~\ref{prop_semi} the semi-trivial solution $(z_\mu^{(1)},0)$ is a saddle point of $\mathcal{J}_{\nu}{\big|}_{\mathcal{N}_\nu}$ and $\tilde{c}_\nu<\mathcal{J}_\nu(z_\mu^{(1)},0)=\min \{\mathscr{C}(\lambda_1,t,s),\mathscr{C}(\lambda_2,t,s)\}$, with $\tilde{c}_\nu $ given in \eqref{cmin}. In the subcritical regime $\alpha+\beta<2^*_s(t)$, because of Lemma~\ref{lemmaPS2} we get the existence of $(\tilde{u},\tilde{v})\in\mathcal{N}_\nu$ with $\tilde{c}_\nu=\mathcal{J}_\nu(\tilde{u},\tilde{v})$. Using the same argument as in Theorem \ref{thm:nugrande} we can assume that $\tilde{u},\tilde{v}\gneq0$. By using the maximum principle \cite[Theorem 1.2]{DelPezzo2017}, we conclude that the pair $(\tilde{u},\tilde{v})$ is a positive ground state solution to system \ref{system}. In the critical regime $\alpha+\beta=2^*_s(t)$, we use Lemma~\ref{lemcritic} to deduce the existence of a positive ground state solution $(\tilde{u},\tilde{v})$ to \eqref{system}. The proof under condition $ii)$ follows similarly.
\end{proof}

As a consequence of the characterization of the semi-trivial solutions provided in Proposition \ref{prop_semi}, those semi-trivial solutions are the ground state if an strict order between the energies $\mathscr{C}(\lambda_1,t,s)$ and $\mathscr{C}(\lambda_2,t,s)$ holds, $\alpha\geq2$ or $\beta\geq0$ and $\nu>0$ small enough.

\begin{theorem}\label{thm:groundstatesalphabeta}
Assume either $\alpha+\beta<2^*_s(t)$ or $\alpha+\beta=2^*_s(t)$ satisfying \eqref{H_1}. Then, 
\begin{itemize}
\item[$i)$] If $\alpha\ge 2$ and $\mathscr{C}(\lambda_1,t,s)>\mathscr{C}(\lambda_2,t,s)$, then there exists $\tilde{\nu}>0$ such that, for any $0<\nu<\tilde{\nu}$, the couple $(0,z_\mu^{(2)})$ is the ground state of \eqref{system}.
\item[$ii)$]  If $\beta\ge 2$ and $\mathscr{C}(\lambda_2,t,s)>\mathscr{C}(\lambda_1,t,s)$, then there exists $\tilde{\nu}>0$ such that for any $0<\nu<\tilde{\nu}$  the couple $(z_\mu^{(1)},0)$ is the ground state of \eqref{system}.
\item[$iii)$] In particular, if $\alpha,\beta \ge 2$, then there exists $\tilde{\nu}>0$ such that for any $0<\nu<\tilde{\nu}$, the couple $(0, z_\mu^{(2)})$ is a ground state of \eqref{system} if $\mathscr{C}(\lambda_1,t,s)>\mathscr{C}(\lambda_2,t,s)$, whereas $(z_\mu^{(1)},0)$ is a ground state otherwise.
\end{itemize}
\end{theorem}
\begin{proof}[Proof of Theorem~\ref{thm:groundstatesalphabeta}]
Assuming hypothesis $i)$, from Proposition~\ref{prop_semi}, it follows that the semi-trivial solution $(0,z_\mu^{(2)})$ is a local minimum for $\nu$ small enough so that $\tilde{c}_{\nu} \leq  \mathcal{J}_{\nu} (0,z_\mu^{(2)})$. Let us prove that the equality actually holds. By contradiction, assume that there exists a sequence $\{\nu_n\} \searrow 0$ such that $\tilde{c}_{\nu_n} <  \mathcal{J}_{\nu_n} (0,z_\mu^{(2)})$. Moreover,
\begin{equation}\label{groundstates1}
\tilde{c}_{\nu_n}< \min \{ \mathscr{C}(\lambda_1,t,s), \mathscr{C}(\lambda_2,t,s) \}= \mathscr{C}(\lambda_2,t,s),
\end{equation}
for $\tilde{c}_{\nu}$ given in \eqref{cmin} with $\nu=\nu_n$. As $\alpha+\beta\leq2^*_s(t)$, the PS condition holds at level $\tilde{c}_{\nu_n}$ by Lemma~\ref{lemmaPS2} and Lemma~\ref{lemcritic}. Then, there exists $(\tilde{u}_n,\tilde{v}_n) \in \mathbb{D}$ with $\tilde{c}_{\nu_n}=\mathcal{J}_{\nu_n} (\tilde{u}_n,\tilde{v}_n)$. 
Repeating the argument used in Theorem \ref{thm:nugrande}, we can assume that $\tilde{u}_n,\tilde{v}_n\geq0$. Even more, the cases $\tilde{u}_n\equiv0$ and $\tilde{v}_n\equiv0$ in $\mathbb{R}^N$ produce a contradiction with \eqref{groundstates1}. Then, by the maximum principle \cite[Theorem 1.2]{DelPezzo2017}, we have $\tilde{u}_n,\tilde{v}_n>0$ in $\mathbb{R}^N\backslash\{0\}$. To continue note that, because of \eqref{NN2}, 
\begin{equation}\label{groundstates3}
\begin{split}
\tilde{c}_{\nu_n} &= \mathcal{J}_{\nu_n} (\tilde{u}_n,\tilde{v}_n)\\
&= \frac{2s-t}{2(N-t)} \left( \int_{\mathbb{R}^N} \dfrac{\tilde{u}_n^{2^*_{s}(t)}}{|x|^{t}}dx+ \int_{\mathbb{R}^N} \dfrac{\tilde{v}_n^{2^*_{s}(t)}}{|x|^{t}}dx\right) + \nu_n \left( \frac{\alpha+\beta-2}{2}  \right)\int_{\mathbb{R}^N} h(x)  \,  \dfrac{\tilde{u}_n^{\alpha} \,  \tilde{v}_n^\beta } {|x|^{t}} \, dx,
\end{split}
\end{equation}
so that, by \eqref{groundstates1}, we deduce that
\begin{equation}\label{groundstates4}
\frac{2s-t}{2(N-t)} \left( \int_{\mathbb{R}^N} \dfrac{\tilde{u}_n^{2^*_{s}(t)}}{|x|^{t}}dx+ \int_{\mathbb{R}^N} \dfrac{\tilde{v}_n^{2^*_{s}(t)}}{|x|^{t}}dx\right)<\mathscr{C}(\lambda_2,t,s)=\frac{2s-t}{2(N-t)} \left[\mathcal{S}(\lambda_2,t,s)\right]^{\frac{N-t}{2s-t}}.
\end{equation}
Since $(\tilde{u}_n,\tilde{v}_n)$ solves \eqref{system}, by using the first equation of \eqref{system} and \eqref{HS_ineq_l}, we find
\begin{equation*}
\mathcal{S}(\lambda_1,t,s) \left( \int_{\mathbb{R}^N} \dfrac{\tilde{u}_n^{2^*_{s}(t)}}{|x|^{t}}dx\right)^{\frac{N-2s}{N-t}} \leq \int_{\mathbb{R}^N} \dfrac{\tilde{u}_n^{2^*_{s}(t)}}{|x|^{t}}dx +  \nu_n \alpha  \int_{\mathbb{R}^N} h(x)  \dfrac{\tilde{u}_n^{\alpha} \,  \tilde{v}_n^\beta } {|x|^{t}} \, dx  .
\end{equation*}
Since, by H\"older's inequality and \eqref{groundstates4}, we have
\begin{equation*}
\begin{split}
\int_{\mathbb{R}^N} h(x)  \,  \dfrac{\tilde{u}_n^{\alpha} \,  \tilde{v}_n^\beta } {|x|^{t}}   \, dx  &\leq C(h)  \left( \int_{\mathbb{R}^N} \dfrac{\tilde{u}_n^{2^*_s(t)}}{|x|^{t}}  \, dx   \right)^{\frac{\alpha}{2^*_{s}(t)}}\left( \int_{\mathbb{R}^N} \dfrac{\tilde{v}_n^{2^*_s(t)}}{|x|^{t}}  \, dx    \right)^{\frac{\beta}{2^*_s(t)}}\\
&\leq C(h) \left(\int_{\mathbb{R}^N} \dfrac{\tilde{u}_n^{2^*_{s}(t)}}{|x|^{t}}dx\right)^{\frac{\alpha}{2}\frac{N-2s}{N-t}} [\mathcal{S}(\lambda_2,t,s)]^{\beta \frac{N-2s}{2(2-t)}}.
\end{split}
\end{equation*}
we conclude
\begin{equation*}
\mathcal{S}(\lambda_1,s) \left(\int_{\mathbb{R}^N} \dfrac{\tilde{u}_n^{2^*_{s}(t)}}{|x|^{t}}dx\right)^{\frac{N-2s}{N-t}} < \sigma_{1,n} + C (h) \nu_n\alpha\, [\mathcal{S}(\lambda_2,t,s)]^{\beta \frac{N-2s}{2(2-t)}} \left(\int_{\mathbb{R}^N} \dfrac{\tilde{u}_n^{2^*_{s}(t)}}{|x|^{t}}dx\right)^{\frac{\alpha}{2}\frac{N-2s}{N-t}}.
\end{equation*}
Applying Lemma~\ref{lem_algebr} with 
\[\sigma=\sigma_{n}=\int_{\mathbb{R}^N} \dfrac{\tilde{u}_n^{2^*_{s}(t)}}{|x|^{t}} \, dx,
\]
there exists $\tilde{\nu}=\tilde{\nu}(\varepsilon)>0$ such that, for any $ 0<\nu_n<\tilde{\nu}$, 
\begin{equation}\label{groundstates5}
\int_{\mathbb{R}^N} \dfrac{\tilde{u}_n^{2^*_{s}(t)}}{|x|^{t}} \, dx> (1-\varepsilon) [\mathcal{S}(\lambda_1,t,s)]^{\frac{N-t}{2s-t}}.
\end{equation}
Since $\mathscr{C}(\lambda_1,t,s)>\mathscr{C}(\lambda_2,t,s)$, we can choose $\varepsilon>0$ such that $(1-\varepsilon) [\mathcal{S}(\lambda_1,t,s)]^{\frac{N-t}{2s-t}} \ge  [\mathcal{S}(\lambda_2,t,s)]^{\frac{N-t}{2s-t}}$ and, hence, from \eqref{groundstates5} we get 
\[
\dfrac{2s-t}{2(N-t)}\int_{\mathbb{R}^N} \dfrac{\tilde{u}_n^{2^*_{s}(t)}}{|x|^{t}} \, dx>\mathscr{C}(\lambda_2,t,s),
\]
in contradiction with \eqref{groundstates4}. Then, 
\begin{equation}\label{groundstates7}
\tilde{c}_\nu =   \frac{2s-t}{2(N-t)} [\mathcal{S}(\lambda_2,t,s)]^{\frac{N-t}{2s-t}},
\end{equation}
for $\nu$ small enough.
Let $(\tilde{u},\tilde{v})$ be a minimizer of $\mathcal{J}_\nu$. If $\tilde{v}\equiv 0$, then a contradiction with \eqref{groundstates7} is found so $\tilde{u}\equiv 0$ and $\tilde{v}$ is such that
\begin{equation*}
(-\Delta)^s \tilde{v} - \lambda_2 \frac{\tilde{v}}{|x|^{2s}}=\dfrac{|\tilde{v}|^{2^*_{s}(t)-2}\tilde{v}}{|x|^{t}} \qquad \mbox{ in } \mathbb{R}^N.
\end{equation*}
Assume that $\tilde{v}$ changes sign so $\tilde{v}^{\pm} \not \equiv 0$ in $\mathbb{R}^N$. Since $(0,\tilde{v}) \in \mathcal{N}_\nu$, then $(0,\tilde{v}^\pm) \in \mathcal{N}_\nu$ and, by \eqref{groundstates3},
\begin{equation*}
\tilde{c}_{\nu}= \mathcal{J}_\nu (0,\tilde{v}) = \frac{2s-t}{2(N-t)} \int_{\mathbb{R}^N} \frac{|\tilde{v}|^{2^*_{s}(t)}}{|x|^{t}}   = \frac{2s-t}{2(N-t)} \int_{\mathbb{R}^N} \left( \frac{|\tilde{v}^+|^{2^*_{s}(t)}}{|x|^{t}}  + \frac{|\tilde{v}^-|^{2^*_s(t)}}{|x|^{t}}\right)  > \mathcal{J}_\nu (0,\tilde{v}^+) \ge  \tilde{c}_{\nu},
\end{equation*}
which is clearly a contradiction. Then, if $\mathscr{C}(\lambda_1,t,s)>\mathscr{C}(\lambda_2,t,s)$, the semi-trivial solution $(0,\pm z_{\mu}^{(2)})$ is the minimizer of $\mathcal{J}_{\nu}{\big|}_{\mathcal{N}_\nu}$ or, equivalently, the semi-trivial solution $(0,z_{\mu}^{(2)})$ is the ground state solution to \eqref{system}. The proof under hypotheses $ii)$ and $iii)$ follows similarly.
\end{proof}

Finally, we found bound states by using a min-max procedure. More precisely, we show that the energy functional admits a Mountain--Pass geometry for certain $\lambda_1,\lambda_2$ verifying a separability condition, that allows us to separate the \textit{semi-trivial} energy levels in a suitable way. The geometry of $\mathcal{J}_\nu$ jointly with the PS condition give us the existence of solution.
\begin{theorem}\label{MPgeom}
Assume either $\alpha+\beta<2^*_s$ or $\alpha+\beta=2^*_s$ satisfying \eqref{H_1}. If
\begin{itemize}
\item[$i)$] Either $\alpha \ge 2$ and
\begin{equation}\label{lamdasalphabeta}
2\mathscr{C}(\lambda_2,s)> \mathscr{C}(\lambda_1,s)>  \mathscr{C}(\lambda_2,s),
\end{equation}
\item[$ii)$] or $\beta \ge 2 $ and
\begin{equation*}
2\mathscr{C}(\lambda_1,s)> \mathscr{C}(\lambda_2,s)>  \mathscr{C}(\lambda_1,s),
\end{equation*}
\end{itemize}
then there exists $\tilde{\nu}>0$ such that for $0<\nu\le \tilde{\nu}$, the system \eqref{system} admits a bound state solution given as a Mountain--Pass-type critical point.
\end{theorem}
The proof of Theorem \ref{MPgeom} will be a direct consequence of the following Lemmas.

\begin{lemma}\label{p1}
Under the hypotheses of Theorem \ref{MPgeom}, there exists $\tilde{\nu}>0$ such that for $0<\nu\le \tilde{\nu}$, the functional $\mathcal{J}_{\nu}{\big|}_{\mathcal{N}_\nu}$ admits a Mountain--Pass geometry.
\end{lemma}
\begin{proof}
Let us assume hypothesis $i)$ and define and the Mountain--Pass level
\begin{equation*}
c_{MP} = \inf_{\psi\in\Psi_\nu} \max_{t\in [0,1]} \mathcal{J}^+_{\nu} (\psi(t)).
\end{equation*}
where $\Psi_\nu$ denotes the set of paths connecting $(z_{\mu}^{(1)},0)$ to $(0,z_{\mu}^{(2)})$ continuously, namely
\begin{equation*}
\Psi_\nu = \left\{ \psi(t)=(\psi_1(\tau),\psi_2(\tau))\in C^0([0,1],\mathcal{N}^+_\nu): \, \psi(0)=(z_1^{(1)},0) \mbox{ and } \, \psi(1)=(0,z_1^{(2)})\right\},
\end{equation*}
Let us take $\psi=(\psi_1,\psi_2) \in \Psi_\nu$, so that, because of \eqref{N1},
\begin{equation}\label{MPgeomp1}
\begin{split}
\|(\psi_1(\tau),\psi_2(\tau))\|^2_{\mathbb{D}}=&\ \int_{\mathbb{R}^N} \frac{(\psi_1^+(\tau))^{2^*_{s}(t)}}{|x|^{t}}dx +\int_{\mathbb{R}^N} \frac{(\psi_2^+(\tau))^{2^*_{s}(t)}}{|x|^{t}} dx\\
&  + \nu(\alpha+\beta) \int_{\mathbb{R}^N} h(x) \frac{(\psi_1^+(\tau))^\alpha (\psi_2^+(\tau))^\beta}{|x|^t} \, dx  ,
\end{split}
\end{equation}
and, because of \eqref{N2},
\begin{equation}\label{MPgeomp2}
\begin{split}
\mathcal{J}^+_{\nu} (\psi(\tau)) =&\ \frac{2s-t}{2(N-t)}\int_{\mathbb{R}^N} \frac{(\psi_1^+(\tau))^{2^*_{s}(t)}}{|x|^{t}} +\int_{\mathbb{R}^N} \frac{(\psi_2^+(\tau))^{2^*_{s}(t)}}{|x|^{t}} dx \\
  &+ \nu\left(\frac{\alpha+\beta-2}{2}\right) \int_{\mathbb{R}^N}h(x)\frac{(\psi_1^+(\tau))^\alpha (\psi_2^+(\tau))^\beta}{|x|^t} \, dx.
\end{split}
\end{equation}
Let us define $\sigma(\tau)\vcentcolon=\left(\sigma_1(\tau),\sigma_2(\tau)\right)$, with 
\[\sigma_j(t)\vcentcolon=\int_{\mathbb{R}^N} \dfrac{(\psi_j^+(\tau))^{2^*_{s}(t)}}{|x|^{t}} \, dx.\]
Observe that, from the definition of $\psi$, we have
\begin{equation*}
\sigma(0)=\left(\int_{\mathbb{R}^N} \frac{(z_1^{(1)})^{2^*_{s}(t)}}{|x|^{t}} \, dx,0\right) \quad \mbox{ and } \quad \sigma(1)=\left(0,\int_{\mathbb{R}^N} \frac{(z_1^{(2)})^{2^*_{s}(t)}}{|x|^{t}} \, dx \right).
\end{equation*}
Using \eqref{HS_ineq_l} and \eqref{MPgeomp1}, 
\begin{equation}\label{MPgeomp3}
\begin{split}
\mathcal{S}(\lambda_1,t,s)(\sigma_1(\tau))^{\frac{N-2s}{N-t}}+\mathcal{S}(\lambda_2,t,s)(\sigma_2(\tau))^{\frac{N-2s}{N-t}}
\leq&\ \|(\psi_1(\tau)\|^2_{\lambda_1},\psi_2(\tau))\|^2_{\mathbb{D}}\\
=&\ \sigma_1(\tau)+\sigma_2(\tau)\\
&+\nu (\alpha+\beta) \int_{\mathbb{R}^N} h \frac{(\psi_1^+(\tau))^\alpha (\psi_2^+(\tau))^\beta}{|x|^t} \, dx,
\end{split}
\end{equation}
while, by H\"older's inequality, 
\begin{equation}\label{MPgeomp4}
\int_{\mathbb{R}^N} h \frac{(\psi_1^+(\tau))^\alpha (\psi_2^+(\tau))^\beta}{|x|^t} \, dx \leq \nu \|h\|_{L^{\infty}}  (\sigma_1(\tau))^{\frac{\alpha}{2}\frac{N-2s}{N-t}} (\sigma_2(\tau))^{\frac{\beta}{2}\frac{N-2s}{N-t}}. 
\end{equation}
Since $\sigma(\tau)$ is continuous function, we can find $\tilde{\tau}\in(0,1)$ for which $\sigma_1(\tilde{\tau})=\sigma_2(\tilde{\tau})=\tilde{\sigma}\neq0$ so that, because of \eqref{MPgeomp3} with $\tau=\tilde{\tau}$ and \eqref{MPgeomp4}, we get 
\begin{equation*}
\left( \mathcal{S}(\lambda_1,t,s) +\mathcal{S}(\lambda_2,t,s)\right) \tilde \sigma^{\frac{2}{2^*_{s}(t)}} \leq 2 \tilde{\sigma} + \nu (\alpha+\beta) \tilde{\sigma}^{\frac{\alpha+\beta}{2^*_{s}(t)}}.
\end{equation*}
Then, due to Lemma~\ref{lem_algebr}, there exists $\tilde{\nu}>0$ small enough such that, for every $0<\nu\le\tilde{\nu}$,\begin{equation}\label{MPgeomp5}
\tilde{\sigma}> \left[\frac{ \mathcal{S}(\lambda_1,t,s)+ \mathcal{S}(\lambda_2,t,s)}{2}\right]^{\frac{N-t}{2s-t}}>  \left[ \mathcal{S}(\lambda_2,t,s) \right]^{\frac{N-t}{2s-t}},
\end{equation}
since, under hypothesis $i)$, $\mathscr{C}(\lambda_1,s)>  \mathscr{C}(\lambda_2,s)$. Therefore, from \eqref{MPgeomp2} and \eqref{MPgeomp5}, we deduce
\begin{equation*}
\max_{\tau\in[0,1]} \mathcal{J}^+_\nu(\psi(\tau)) > 2 \frac{2s-t}{2(N-t)} \left[ \mathcal{S}(\lambda_2,t,s) \right]^{\frac{N-t}{2s-t}} = 2 \mathscr{C}(\lambda_2,t,s)>\mathscr{C}(\lambda_1,t,s).
\end{equation*}
As a consequence, $c_{MP}> \mathscr{C}(\lambda_1,t,s)=\max\{\mathcal{J}^+_\nu(z_1^{(1)},0),\mathcal{J}^+_\nu(z_1^{(2)},0)\}$, and we conclude that $\mathcal{J}_{\nu}{\big|}_{\mathcal{N}_\nu}$ admits a Mountain--Pass geometry. The proof under hypothesis $ii)$ follows analogously.
\end{proof}
Next we address the compactness with the aid of Lemmas~\ref{lemmaPS1} and \ref{lemcritic}. To that end, we prove that the mountain pass level lies in the correct energy range.
\begin{lemma}\label{p2}
Under the hypotheses of Theorem \ref{MPgeom}, the PS condition holds for the Mountain--Pass level $c_{MP}$.
\end{lemma}
\begin{proof}
Let us take the path $ \psi(t) =( \psi_1(\tau),   \psi_2(\tau))=\left((1-\tau)^{1/2} z_\mu^{(1)},\tau^{1/2}z_\mu^{(2)} \right)$ with $\tau\in[0,1]$. By definition of the Nehari manifold, there exists a positive function $\gamma:[0,1]\mapsto(0,+\infty)$ such that $\gamma \psi \in \mathcal{N}_\nu^+\cap \mathcal{N}_\nu$ for  $\tau\in[0,1]$. Note that $\gamma(0)=\gamma(1)=1$. As before, we define
\begin{equation*}
\sigma(\tau)=(\sigma_1(\tau),\sigma_2(\tau))=\left(\int_{\mathbb{R}^N}\frac{\left( \gamma \psi_1(\tau)\right)^{2^*_{s}(t)}}{|x|^{t}} \, dx , \int_{\mathbb{R}^N} \frac{\left( \gamma \psi_2(\tau)\right)^{2^*_{s}(t)}}{|x|^{t}} \, dx \right).
\end{equation*}
By \eqref{relation}, we have that
\begin{equation}\label{Mpgeom7}
\sigma_1(0)=[\mathcal{S}(\lambda_1,t,s)]^{\frac{N-t}{2s-t}}\quad\text{and}\quad \sigma_2(1)=[\mathcal{S}(\lambda_2,t,s)]^{\frac{N-t}{2s-t}}.
\end{equation}
Since $\gamma \psi(t)\in\mathcal{N}^+_\nu\cap \mathcal{N}_\nu$, using \eqref{normH}, we get
\begin{equation*}
\begin{split}
\left\|\left((1-\tau)^{1/2} z_\mu^{(1)},\tau^{1/2}z_\mu^{(2)} \right)\right\|^2_\mathbb{D} &= (1-\tau)\sigma_1(0)+\tau\sigma_2(1) \\
&= \gamma^{2^*_{s}(t)-2}(\tau)\left((1-\tau)^{2^*_{s}(t)/2} \sigma_1(0) + \tau^{2^*_{s}/2}\sigma_2(1)\right)   \\
&\mkern+20mu + \nu (\alpha+\beta)\gamma^{\alpha+\beta-2} (\tau)(1-\tau)^{\alpha/2}\tau^{\beta/2} \int_{\mathbb{R}^N} h(x) \frac{ (z_\mu^{(1)})^\alpha (z_\mu^{(2)})^{\beta}}{|x|^t} dx,
\end{split}
\end{equation*}
implying that, for every  $\tau\in(0,1)$, it holds
\begin{equation}\label{g1}
(1-\tau)\sigma_1(0)+\tau\sigma_2(1)>\gamma^{2^*_{s}(t)-2}(\tau)\left((1-\tau)^{2^*_{s}(t)/2} \sigma_1(0) + \tau^{2^*_{s}(t)/2}\sigma_2(1)\right).
\end{equation}
As $\gamma \psi \in \mathcal{N}_\nu^+$, because of \eqref{Nnueq} and \eqref{g1}, we find
\begin{equation}\label{g2}
\begin{split}
\mathcal{J}_\nu^+ (\gamma \psi(\tau))=&\ \left( \frac{1}{2}-\frac{1}{\alpha+\beta} \right)\|\gamma\psi(\tau) \|^2_\mathbb{D}\\
&+ \left(\frac{1}{\alpha+\beta}- \frac{1}{2^*_{s}(t)} \right) \gamma^{2^*_{s}(t)}(\tau)  \int_{\mathbb{R}^N}\frac{(\psi_1(\tau))^{2^*_{s}(t)}}{|x|^{t}}dx   +\int_{\mathbb{R}^N}\frac{(\psi_2(\tau))^{2^*_{s}(t)}}{|x|^{t}}dx\vspace{0.3cm} \\
 =&\ \gamma ^2(\tau)  \left( \frac{1}{2}-\frac{1}{\alpha+\beta} \right)  \left[(1-\tau) \sigma_1(0) + \tau \sigma_2(1) \right] \\
 &+ \gamma^{2^*_{s}(t)} \left(\frac{1}{\alpha+\beta}- \frac{1}{2^*_{s}(t)} \right) \left[  (1-\tau)^{2^*_{s}(t)/2} \sigma_1(0)   +  \tau^{2^*_{s}(t)/2} \sigma_2(1)\right] \\
<&\ \frac{2s-t}{2(N-t)} \gamma^2(\tau)\left[(1-\tau)\sigma_1(0) + \tau  \sigma_2(1) \right].
\end{split}
\end{equation}
From \eqref{g1} and \eqref{g2}, we deduce that
\begin{equation*}
g(t)\!\vcentcolon=\!\frac{2s-t}{2(N-t)}\!\! \left[ \dfrac{(1-\tau) \sigma_1(0) + \tau \sigma_2(1)}{(1\!-\!\tau)^{2^*_{s}(t)/2} \sigma_1(0)\!+\! \tau^{2^*_{s}/2} \sigma_2(1)}  \right]^{\frac{2}{2^*_{s}-2}} \left[(1-\tau)\sigma_1(0) + \tau  \sigma_2(1) \right]\!\ge\!\max_{\tau\in[0,1]} \mathcal{J}_\nu^+ (\gamma \psi(\tau)).
\end{equation*}
It is easy to see that the function $g$ attains its maximum value at $\tau=\frac{1}{2}$. Indeed, because of \eqref{Mpgeom7}, 
\begin{equation*}
g\left(\dfrac{1}{2}\right)=\frac{2s-t}{2(N-t)} \left(\sigma_1(0) +  \sigma_2(1)\right) = \mathscr{C}(\lambda_1,t,s)+\mathscr{C}(\lambda_2,t,s).
\end{equation*}
Then, using \eqref{g2} and  \eqref{lamdasalphabeta}, we get  
\[
\mathcal{J}_\nu^+ (\gamma \psi(\tau))<  \mathscr{C}(\lambda_1,t,s)+\mathscr{C}(\lambda_2,t,s) < 3 \mathscr{C}(\lambda_2,t,s),
\]
and, consequently,
\[\mathscr{C}(\lambda_2,t,s)<\mathscr{C}(\lambda_1,t,s)<c_{MP} \leq \max_{\tau\in[0,1]} \mathcal{J}_\nu^+(\gamma \psi(\tau))<  3 \mathscr{C}(\lambda_2,t,s).
\]
Then, the  Mountain--Pass level $c_{MP}$ satisfies the assumptions of Lemmas~\ref{lemmaPS1} and \ref{lemcritic}.
\end{proof}
\begin{proof}[Proof of Theorem~\ref{MPgeom}]
Taking in mind Lemmas \ref{p1} and \ref{p2} above, by the Mountain-Pass theorem, the functional $\mathcal{J}^+_\nu|_{\mathcal{N}
^+_\nu}$ has a PS sequence $\{(u_n,v_n)\}$ at level $c_{MP}$. If $\alpha+\beta<2^*_s(t)$, adapting Lemma~\ref{lem_PSNehari} for $\mathcal{J}^+_\nu$ combined with Lemma \ref{lemmaPS1}, it follows that the sequence $\{(u_n,v_n)\}$ has a subsequence which strongly converges to a critical point $(\tilde{u},\tilde{v})$ of the truncated functional $\mathcal{J}^+_\nu$ on the Nehari manifold $\mathcal{N}^+_\nu$. Indeed, $(\tilde{u},\tilde{v} )$ is also a critical point of $\mathcal{J}_\nu$ on $\mathcal{N}_\nu$ and, thus, it is also a critical point of $\mathcal{J}_\nu$ defined in $\mathbb{D}$. Moreover, we have $\tilde{u},\tilde{v}\geq0$ in $\mathbb{R}^N\backslash\{0\}$ and $\tilde{u},\tilde{v}\not\equiv0$. Fuerthermore, by the maximum principle \cite[Theorem 1.2]{DelPezzo2017} we conclude $\tilde{u},\tilde{v}\not\equiv0$ in $\mathbb{R}^N\backslash\{0\}$. Hence, $( \tilde{u},\tilde{v} )$ is a bound state solution to the system \ref{system}. In the critical regime $\alpha+\beta=2_s^*(t)$, the result follows by proving the compactness part using Lemma~\ref{lemcritic}.
\end{proof}


\end{document}